\documentclass[11pt,a4paper,reqno]{amsart}

\usepackage{amsmath,amssymb,mathtools,esint}

\usepackage{enumitem}
\usepackage{hyperref}

\usepackage[left=1in,right=1in,top=1in,bottom=1in]{geometry}
\hypersetup{
  colorlinks=true,
  linkcolor=blue,
  citecolor=blue,
  urlcolor=blue,
}

\numberwithin{equation}{section}

\newtheorem{theorem}{Theorem}[section]
\newtheorem{lemma}[theorem]{Lemma}
\newtheorem{proposition}[theorem]{Proposition}
\newtheorem{corollary}[theorem]{Corollary}
\theoremstyle{definition}
\newtheorem{definition}[theorem]{Definition}
\theoremstyle{remark}
\newtheorem{remark}[theorem]{Remark}

\newcommand{\R}{\mathbb{R}}
\newcommand{\Sph}{\mathbb{S}}
\newcommand{\Tr}{\operatorname{Tr}}

\newcommand{\dashint}{\fint}

\title[Boundary regularity for Oseen--Frank minimizers]
{Boundary Regularity for Oseen--Frank Minimizers with Arbitrary Positive
Splay, Twist, and Bend Constants}

\author{Zhiyuan Dai}
\address{School of Mathematical Sciences, Zhejiang University, Hangzhou 310058, China}
\email{12635038@zju.edu.cn}

\author{Haotong Fu}
\address{School of Mathematical Sciences, Peking University, Beijing 100871, China}
\email{2301110012@pku.edu.cn}

\author{Huaijie Wang}
\address{School of Mathematical Sciences, Peking University, Beijing 100871, China}
\email{huaijie\_wang@163.com}

\author{Wei Wang}
\address{School of Mathematical Sciences, Peking University, Beijing 100871, China}
\email{wwmath166@outlook.com}
\email{2201110024@stu.pku.edu.cn}

\date{}

\begin{document}
\begin{abstract}
Let \(\Omega\subset\mathbb R^3\) be a smooth bounded domain and let
\(g:\partial\Omega\to\mathbb S^2\) be smooth. We prove that any global
minimizer of the Oseen--Frank energy subject to the strong anchoring condition
\(n=g\) is smooth in a full neighborhood of \(\partial\Omega\). The result
holds for arbitrary positive splay, twist, and bend constants, without any
small-anisotropy assumption. It resolves the boundary-regularity part of Lin and Liu's Problem~(c) for the pure Oseen--Frank problem and for its
prescribed smooth magnetic-field perturbation.
\end{abstract}


\maketitle

\section{Introduction}

\subsection{Variational setting}
\label{subsec:of-variational-problem}

Let \(\Omega\subset\mathbb R^3\) be a smooth bounded domain occupied by a
nematic liquid crystal. The classical theory originates with
Oseen~\cite{Oseen1933} and Frank~\cite{Frank1958}; standard expositions
include de Gennes and Prost~\cite{deGennesProst1993} and
Virga~\cite{Virga1994}. Its orientation is described by a director field
\[
n:\Omega\to\mathbb S^2,
\quad | n|=1\quad\text{a.e. in }\Omega.
\]
The elastic energy is
\begin{align}
2E_{\mathrm{OF}}(n;\Omega)
=\int_\Omega\Bigl\{&
 k_1(\operatorname{div}n)^2
 +k_2(n\cdot\operatorname{curl}n)^2
 +k_3| n\times\operatorname{curl}n|^2
 \notag\\
&+(k_2+k_4)
 \bigl[
   \operatorname{tr}((\nabla n)^2)
   -(\operatorname{div}n)^2
 \bigr]
\Bigr\}\,\mathrm{d} x.
\end{align}
Here, \(k_1,k_2,k_3>0\) are the splay, twist, and bend elastic constants,
respectively, while \(k_4\) is the saddle-splay constant. The last integrand is
a null Lagrangian. Consequently, its integral is determined by the Dirichlet
trace and does not affect minimizers in a fixed strong-anchoring class.

Given a smooth boundary datum
\[
g:\partial\Omega\to\mathbb S^2,
\]
define
\[
\mathcal A_g
:=
\left\{
v\in H^1(\Omega;\mathbb S^2):
\operatorname{Tr}v=g
\right\}.
\]
The problem considered in this paper is
\[
E_{\mathrm{OF}}(n;\Omega)
=
\min_{v\in\mathcal A_g}E_{\mathrm{OF}}(v;\Omega).
\]
Thus \(n\) is a global minimizer subject to the strong anchoring condition
\(n=g\) on \(\partial\Omega\).

Set
\[
\alpha:=\min\{k_1,k_2,k_3\}.
\]
Modulo the null Lagrangian, the preceding problem is equivalent to minimizing
the coercive density
\[
\begin{aligned}
2W_\alpha(z,P)
={}&
\alpha| P|^2
+(k_1-\alpha)(\operatorname{tr}P)^2\\
&+(k_2-\alpha)(z\cdot\operatorname{curl}P)^2
 +(k_3-\alpha)| z\times\operatorname{curl}P|^2.
\end{aligned}
\]
In particular,
\begin{equation}
\lambda| P|^2
\leq W_\alpha(z,P)
\leq\Lambda| P|^2,
\label{eq:coercivity}
\end{equation}
where \(0<\lambda\leq\Lambda<+\infty\) depend only on the Frank constants.
The density \(W_\alpha\) is autonomous, real analytic in \(z\),
two-homogeneous in \(P\), and uniformly strictly convex in \(P\).

The direct method gives existence of minimizers, while the foundational work
of Hardt, Kinderlehrer, and Lin~\cite{HKL1986,HKL1988} established interior
partial regularity, higher integrability, and quantitative density bounds for
static liquid-crystal configurations. In particular, a minimizer is smooth
away from a relatively closed interior singular set of Hausdorff dimension
strictly less than one. For smooth boundary data, the existing partial
boundary theory yields a closed set
\(\Sigma_1\subset\partial\Omega\), with \(\mathcal H^1(\Sigma_1)=0\),
outside which the minimizer is regular; see Lin and
Wang~\cite[Theorem~1.1]{LinWang2014}. That theory does not show that
\(\Sigma_1\) is empty, which is the gap addressed here.

\subsection{Previous results and the boundary problem}

When \(k_1=k_2=k_3\), the Oseen--Frank energy, modulo its null Lagrangian,
is a constant multiple of the Dirichlet energy. The minimizers are then
energy-minimizing harmonic maps into \(\mathbb S^2\), for which Schoen and
Uhlenbeck~\cite{SU1982,SU1983} proved interior and boundary regularity
theorems; see also Simon~\cite{Simon1996}. Their boundary blow-up
argument uses strong compactness, small-energy decay, and monotonicity of the
normalized Dirichlet energy. Monotonicity makes a half-space tangent map
homogeneous, after which a rigidity argument excludes non-constant tangents
with constant boundary trace.

For unequal Frank constants, the energy remains quadratic and uniformly
elliptic but is spatially anisotropic. The radial inner-variation identity no
longer reduces to a non-negative square, so the harmonic-map monotonicity
argument is unavailable. Lin and Liu~\cite[p.~302]{LinLiu2001} isolated this obstruction and asked
whether the normalized energy is monotone, whether singularities are
isolated, and whether a minimizer with smooth strong-anchoring data is smooth
near the boundary. The last question is the
boundary-regularity part of their Problem~(c).

The following theorem answers that boundary question directly, without
assuming or proving a normalized-energy monotonicity formula.

\begin{theorem}[Boundary regularity]\label{thm:main}
Let \(\Omega\subset\mathbb R^3\) be a smooth bounded domain and let
\(g:\partial\Omega\to\mathbb S^2\) be smooth. For arbitrary
\[
k_1,k_2,k_3>0,
\]
any global minimizer \(n\in\mathcal A_g\) of the Oseen--Frank energy is
smooth in a full neighborhood of \(\partial\Omega\). Equivalently, there
exists \(\rho>0\) such that
\[
n\in C^\infty(\overline{\Omega_\rho};\mathbb S^2),
\quad
\Omega_\rho
:=
\left\{
x\in\Omega:
\operatorname{dist}(x,\partial\Omega)<\rho
\right\}.
\]
In particular,
\[
\operatorname{Sing}n\cap\partial\Omega=\varnothing.
\]
\end{theorem}

No smallness or closeness assumption is imposed on the ratios of the Frank
constants.

Lin and Liu formulated Problem~(c) simultaneously for their Problems~1, 3,
and 4. Theorem~\ref{thm:main} settles the boundary question for Problem~1.
Their Problem~3 adds a prescribed smooth magnetic field, which contributes a
zeroth-order potential and therefore disappears under boundary blow-up. The
same observation also permits a chiral first-order perturbation.

\begin{corollary}[Chirality and smooth lower-order fields]

The conclusion of Theorem~\ref{thm:main} remains valid after adding to the
integrand the chiral term
\[
k_2\bigl\{
(n\cdot\operatorname{curl}n+q_0)^2
-(n\cdot\operatorname{curl}n)^2
\bigr\},
\quad q_0\in\mathbb R,
\]
and any potential
\[
V\in C^\infty(\overline\Omega\times\mathbb S^2).
\]
In particular, the conclusion applies to the prescribed-magnetic-field energy
in Problem~3 of Lin and Liu~\cite{LinLiu2001}.
\end{corollary}

\begin{proof}
Write \(E_W\) for the coercive quadratic part. If \(u\) minimizes the
perturbed energy in \(B_r\) and \(v=u\) on \(\partial B_r\), the first-order
term, Young's inequality, and the boundedness of the smooth potential give
\begin{equation*}
E_W(u;B_r)
\leq
E_W(v;B_r)
+Cr\bigl(E_W(u;B_r)+E_W(v;B_r)\bigr)+Cr^2.
\end{equation*}
For sufficiently small \(r\), the term containing \(E_W(u;B_r)\) is
absorbed into the left-hand side. Thus \(u\) is a quadratic-energy
quasi-minimizer with multiplicative error \(1+Cr\) and additive error
\(Cr^2\).

This inequality supplies each perturbative estimate used below. The filling
comparison gives the same uniform upper density bound with an additional
\(Cr\) term. The compactness splice, divided by the three-dimensional energy
scale \(r\), has an error tending to zero. The small-energy iteration becomes
\[
\Theta_u(a,\theta r)
\leq
\frac{1}{2}\Theta_u(a,r)+Cr,
\]
whose error is summable on dyadic scales; hence it gives both clearing-out and
the singular-point lower density bound. Finally, in any blow-up sequence the
quasi-minimality errors vanish, so the limit is an exact class-A minimizer of
the autonomous quadratic density. The stress-measure and Liouville parts of
the proof therefore apply without change. Smoothness follows from the smooth
coefficients of the chiral term and \(V\).
\end{proof}

Lin and Liu's Problem~4~\cite{LinLiu2001}, which couples the director field to an electric
potential through a Maxwell constraint, is not covered here. Eliminating the
electric potential generally produces a non-local functional of the director,
whereas the present argument relies on the locality and quadratic-gradient
structure of the Oseen--Frank energy. Accordingly, our result resolves the
Lin--Liu boundary-regularity question for Problems~1 and 3, but makes no claim
about Problem~4.

\subsection{Strategy of the proof}

Assume, for contradiction, that a minimizer has a singular point
\(a\in\partial\Omega\). After flattening the boundary and taking a blow-up at
\(a\), the compactness and clearing-out theory produces a non-constant class-A
minimizer
\[
U:\mathbb R^3_+\to\mathbb S^2
\]
of an autonomous analytic quadratic density. Its trace on the flat boundary
is constant, and it satisfies the two-sided growth estimate
\[
cR
\leq
E(U;B_R^+)
\leq
CR
\quad R>0.
\]
The lower bound records the non-triviality inherited from the original
singular point; the upper bound is the uniform density estimate.

The replacement for monotonicity is obtained from one-sided inner variations.
Extend \(U\) constantly across the flat boundary. Variations pushing points
towards the constant side remain admissible, whereas variations in the
opposite direction need not. This one-sided minimality shows that the normal
column of the energy-momentum tensor defines a non-negative Radon measure
\(\mu_U\) on \(\partial\mathbb R^3_+\). On any smooth boundary patch,
\[
\mathrm{d} \mu_U
=
W(q,\partial_3U\otimes e_3)\,\mathrm{d} x',
\]
so that
\[
\mu_U=0
\quad\Leftrightarrow\quad
\partial_3U=0
\]
on that patch.

The linear upper energy bound implies
\[
\mu_U(\mathbb R^2)<+\infty.
\]
We then perform a blow-down. Namely, for
\(R_j\to+\infty\), set
\[
U_j(x):=U(R_jx).
\]
In dimension three the boundary stress has the exact scaling law
\[
\mu_{U_j}(A)=\mu_U(R_jA).
\]
Since \(\mu_U\) has finite total mass, any compact subset of
\(\mathbb R^2\backslash\{0\}\) is eventually mapped to an annulus escaping to
infinity. Hence the boundary stress of the blow-down limit \(V\) is supported
at the origin.

Partial boundary regularity provides a regular boundary disk away from the
origin. On that disk,
\[
V=q,
\quad
\partial_3V=0.
\]
Thus \(V\) has both constant Dirichlet data and vanishing normal derivative.
Analytic Cauchy uniqueness implies that \(V=q\) on a non-empty interior open
set. The interior singular set has Hausdorff dimension strictly less than
one, so its complement is connected. Analytic continuation therefore gives
\[
V\equiv q
\quad\text{a.e. in }\mathbb R^3_+.
\]
This contradicts the positive energy lower bound inherited under the
blow-down.

The dimension-three hypothesis enters critically at the large-scale stage.
For a quadratic energy in domain dimension \(m\), the natural energy growth
is \(R^{m-2}\), while the boundary stress satisfies
\[
\mu_{U_R}(A)=R^{3-m}\mu_U(RA).
\]
Only when \(m=3\) is this measure scale invariant and controlled by a finite
total mass under the natural energy bound. This critical scaling is what
forces the blow-down stress to concentrate at a single boundary point.

\section{Boundary compactness and small-energy theory}

\subsection{Geometric setting and terminology}

Throughout this section,
\[
B_r^+:=B_r\cap\{x_3>0\},\quad
\Gamma_r:=B_r\cap\{x_3=0\},\quad
S_r^+:=\partial B_r\cap\{x_3>0\}.
\]
The set \(\Gamma_r\) is the flat Dirichlet face and \(S_r^+\) is the open hemispherical part of the boundary. All constants depend only on the ellipticity ratio, the indicated coefficient bounds, and the target \(\Sph^2\), unless additional dependence is stated explicitly.

A map \(u\in H^1(B_r^+;\Sph^2)\) is called a local minimizer if, for any Lipschitz subdomain \(G\subset\subset B_r^+\cup\Gamma_r\) and any \(v\in H^1(G;\Sph^2)\) having the same trace as \(u\) on the relative boundary of \(G\), one has
\[
E_W(u;G)\leq E_W(v;G).
\]
A point is regular if \(u\) is H\"older continuous in a neighbourhood of
that point, up to the flat face when the point lies on \(\Gamma_r\);
otherwise it is singular. Once a point is regular, the Euler system gives
the regularity allowed by the coefficients and boundary data; in particular,
regular points are smooth when those data are smooth. In dimension three the
scale-invariant energy is
\[
\Theta_u(a,r):=r^{-1}E_W\bigl(u;B_r(a)\cap\Omega\bigr),
\]
because a quadratic gradient energy scales like \(r^{3-2}=r\).

\subsection{Blow-up densities and the compactness package}

Boundary flattening and rescaling transform the coercive Oseen--Frank density
\(W_\alpha\) into a family
\[
W_j(y,z,P)
=
J_j(y)W_\alpha\bigl(z,PB_j(y)\bigr),
\]
where
\[
J_j(y):=\lvert\det DF(r_jy)\rvert,
\quad
B_j(y):=DF(r_jy)^{-1}.
\]
Here \(F\) is a boundary-flattening diffeomorphism satisfying
\[
F(0)=0,
\quad
DF(0)=I,
\]
and \(r_j\downarrow0\). Consequently,
\[
J_j\to1,
\quad
B_j\to I,
\quad
W_j\to W_\alpha
\]
locally in \(C^1\). Thus the rescaled densities are generally
\(x\)-dependent, whereas their blow-up limit is autonomous.

The following lemma collects the compactness and regularity properties used
later.

\begin{lemma}[Uniform boundary blow-up package]
\label{lem:package}
For each \(j\), let
\[
W_j(x,z,P)
=
\frac{1}{2}
A_{j,ik}^{ab}(x,z)P_i^aP_k^b
\]
be a \(C^2\) density on
\[
B_2^+\times\mathbb S^2\times\mathbb R^{3\times3}.
\]
Assume that
\[
\lambda| P|^2
\leq
A_{j,ik}^{ab}(x,z)P_i^aP_k^b
\leq
\Lambda| P|^2,
\quad
0<\lambda\leq\Lambda,
\]
that the \(C^2\) norms of \(A_j\) are uniformly bounded, and that
\[
A_j\to A_0
\quad\text{in }C^1,
\]
where \(A_0=A_0(z)\) is independent of \(x\).

Let \(u_j\in H^1(B_2^+;\mathbb S^2)\) be local minimizers of \(E_{W_j}\),
with flat-boundary traces
\[
g_j\to q\in\mathbb S^2
\quad\text{in }C^2,
\]
and suppose that
\[
\sup_j\int_{B_2^+}|\nabla u_j|^2\,\mathrm{d} x<+\infty.
\]
After passing to a subsequence, the following conclusions hold.

\begin{enumerate}[label=$(\theenumi)$]
\item\label{Strong compactness} \emph{Strong compactness.}
There exists \(u_0\in H^1(B_2^+;\mathbb S^2)\) such that
\[
u_j\to u_0
\quad\text{strongly in }H^1(B_s^+)
\]
for any \(s<2\). Moreover,
\[
\operatorname{Tr}_{\Gamma_2}u_0=q,
\]
and \(u_0\) locally minimizes
\[
W_0(z,P)
=
\frac{1}{2}A_{0,ik}^{ab}(z)P_i^aP_k^b.
\]

\item\label{Energy decay and clearing-out} \emph{Energy decay and clearing-out.}
In dimension three there exist
\[
\varepsilon_*>0,
\quad
\vartheta\in\left(0,\frac{1}{4}\right),
\]
depending only on the structural bounds, such that
\[
\Theta_u(a,r)<\varepsilon_*
\quad\Rightarrow\quad
\Theta_u(a,\vartheta r)
\leq
\frac{1}{2}\Theta_u(a,r)
\]
in the autonomous constant-trace case. Iteration gives H\"older continuity at
\(a\); if the coefficients and boundary data are smooth, elliptic
bootstrapping gives smoothness.
Consequently,
\[
a\in\operatorname{Sing}u
\quad\Rightarrow\quad
\Theta_u(a,r)\geq\varepsilon_*
\]
for any sufficiently small \(r\).

For smoothly flattened boundaries and smooth non-constant traces, the decay
estimate contains lower-order errors that vanish under blow-up.

\item\label{Upper density bound} \emph{Upper density bound.}
There exists \(M<+\infty\) such that
\[
E_W\bigl(u;B_r(a)\cap\Omega\bigr)\leq Mr
\]
at any \(a\in\overline\Omega\) and any sufficiently small \(r\). The
constant is uniform along the preceding blow-up sequences.

\item\label{Partial regularity} \emph{Partial regularity.}
The interior singular set satisfies
\[
\dim_{\mathcal H}\operatorname{Sing}_{\mathrm{int}}u<1.
\]
The map is smooth on the regular set when the coefficients are smooth, and is
real analytic when the limiting autonomous coefficients are real analytic in
\(z\). If the flat trace is constant, any boundary disk contains a non-empty
relatively open regular patch.
\end{enumerate}

The same conclusions hold, after pullback, on smoothly converging flattened
domains.
\end{lemma}

\begin{proof}
For the property \ref{Strong compactness}, the energy bound gives, after taking a subsequence,
\[
u_j\rightharpoonup u_0
\quad\text{in }H^1(B_2^+),
\quad
u_j\to u_0
\quad\text{in }L^2(B_2^+).
\]
The trace theorem gives
\[
\operatorname{Tr}u_0=q.
\]
Lemma~\ref{lem:luckhaus} joins a competitor for \(u_0\) to \(u_j\) in a thin
half-annulus while preserving the flat trace. The joining energy tends to
zero. Minimality and weak lower semicontinuity then give
\[
E_{W_0}(u_0;B_s^+)
\leq
E_{W_0}(v;B_s^+)
\]
for any admissible competitor \(v\). Repeating the joining argument with
\(v=u_0\) gives convergence of the energies. Uniform strict convexity then
yields
\[
\nabla u_j\to\nabla u_0
\quad\text{in }L^2(B_s^+).
\]
This proves \ref{Strong compactness}; see Proposition~\ref{prop:compactness}.

For \ref{Energy decay and clearing-out}, Proposition~\ref{prop:improvement} proves the one-scale decay by
contradiction. If the estimate failed, maps with energies
\[
\epsilon_j^2
:=
\int_{B_1^+}|\nabla u_j|^2
\to0
\]
could be normalized by
\[
w_j:=\frac{u_j-q_j}{\epsilon_j}.
\]
Part \ref{Strong compactness} and the reverse H\"{o}lder estimate give strong convergence to a
\(T_q\mathbb S^2\)-valued solution of a constant-coefficient linear elliptic
system with zero flat trace. Boundary estimates for the limiting system give
\[
\rho^{-1}\int_{B_\rho^+}|\nabla w|^2
\leq
C\rho^2\int_{B_{\frac{3}{4}}^+}|\nabla w|^2,
\]
contradicting the assumed failure of decay. Iteration gives, for some
\(\beta>0\),
\[
E_W\bigl(u;B_\rho(a)\cap\Omega\bigr)
\leq
C\rho^{1+2\beta}.
\]
Morrey's lemma and elliptic bootstrapping imply regularity. Taking the
contrapositive gives the singular-point lower density bound.

For \ref{Upper density bound}, coarea provides \(s\in(\frac{r}{2},r)\) such that
\[
\int_{\partial B_s^+}|\nabla_{\tan}u|^2
\leq
C\Theta_u(a,r).
\]
Using Lemma~\ref{lem:filling} to fill the trace on
\(\partial B_s^+\), minimality gives
\[
\Theta_u(a,\frac{r}{2})
\leq
C_0\sqrt{\Theta_u(a,r)}+o(1).
\]
If \(\Theta_u(a,r)\) is large, the right-hand side is at most a fixed fraction
of \(\Theta_u(a,r)\); if it is bounded, the recurrence preserves a uniform
bound. Dyadic iteration proves
\[
\Theta_u(a,r)\leq M,
\]
which is equivalent to \ref{Upper density bound}. This is Lemma~\ref{lem:upperdensity}.

For the property \ref{Partial regularity}, the reverse H\"older estimate gives
\[
\nabla u\in L^p_{\mathrm{loc}}
\quad\text{for some }p>2.
\]
Combining this with the singular-point lower density bound and a covering
argument yields
\[
\dim_{\mathcal H}\operatorname{Sing}_{\mathrm{int}}u
\leq
3-p<1.
\]
On the regular set, energy decay initiates the standard elliptic bootstrap.
Smooth or analytic coefficients give the corresponding regularity.

Finally, let
\[
\nu:=|\nabla u|^2\,\mathrm{d} x.
\]
The Vitali argument in Subsection~\ref{subsec:singular-set} shows that
\[
\left\{
a\in\partial\mathbb R^3_+:
\limsup_{r\downarrow0}r^{-1}\nu(B_r(a))>0
\right\}
\]
has zero two-dimensional Hausdorff measure. Hence, any boundary disk contains
a point where the normalized energy tends to zero. Part \ref{Energy decay and clearing-out} gives regularity
at that point, and openness of the regular set gives a non-empty relatively
open regular patch.

Smoothly converging domains are reduced to the preceding argument by pullback.
The transformed coefficients retain uniform ellipticity and uniform \(C^2\)
bounds.
\end{proof}

\begin{remark}
The boundary conclusion in Lemma~\ref{lem:package}\ref{Partial regularity} gives only one non-empty
regular patch in each boundary disk. It does not assert full boundary
regularity and therefore does not use Theorem~\ref{thm:main}.
\end{remark}

\subsection{Relative interpolation and half-ball filling}

We use the following relative form of the Luckhaus interpolation lemma. Its last clause is the point that permits strong convergence on sets meeting the flat face.

\begin{lemma}[Flat-boundary Luckhaus interpolation]
\label{lem:luckhaus}
Let
\[
A_\delta^+
:=
B_1^+\backslash B_{1-\delta}^+,
\quad
E:=\partial S_1^+,
\quad
0<\delta<\frac{1}{8}.
\]
Suppose that
\[
v,w\in H^1(S_1^+;\mathbb S^2)
\]
have the same constant trace \(q\in\mathbb S^2\) on \(E\). Then there exists
\[
z\in H^1(A_\delta^+;\mathbb S^2)
\]
such that
\[
z(\omega)=v(\omega),
\quad
z((1-\delta)\omega)=w(\omega)
\quad\text{for a.e. }\omega\in S_1^+,
\]
and
\[
z=q
\quad\text{on }
A_\delta^+\cap\{x_3=0\}.
\]
Moreover,
\begin{equation}
\begin{aligned}
\int_{A_\delta^+}|\nabla z|^2\,\mathrm{d} x
\leq{}&
C\delta
\int_{S_1^+}
\bigl(
|\nabla_{\tan}v|^2+
|\nabla_{\tan}w|^2
\bigr)\,\mathrm{d} S +
C\delta^{-1}
\int_{S_1^+}| v-w|^2\,\mathrm{d} S.
\end{aligned}
\label{eq:luckhaus-estimate}
\end{equation}
The constant \(C\) is independent of \(v,w\), and \(\delta\).

The same conclusion holds, with uniform constants, for a family of
\(C^1\) half-balls converging in \(C^1\). More generally, if
\[
\operatorname{Tr}_E v=g_v,
\quad
\operatorname{Tr}_E w=g_w,
\quad
g_v,g_w\in H^1(E;\mathbb S^2),
\]
then the flat trace of \(z\) may be chosen to join \(g_v\) to \(g_w\), and
the right-hand side of \eqref{eq:luckhaus-estimate} acquires the additional
terms
\[
C\delta^{-1}
\| g_v-g_w\|_{L^2(E)}^2
+
C\delta
\left(
\|\nabla_{\tan}g_v\|_{L^2(E)}^2
+
\|\nabla_{\tan}g_w\|_{L^2(E)}^2
\right).
\]
\end{lemma}

\begin{proof}
The proof is the relative version of the interpolation construction of
Luckhaus~\cite[Lemma~1]{Luckhaus1988}.

\medskip
\noindent
\emph{Product representation.}
The map
\[
\Phi:S_1^+\times(0,\delta)\rightarrow A_\delta^+,
\quad
\Phi(\omega,t):=(1-t)\omega,
\]
is bi-Lipschitz, with constants independent of
\(0<\delta<\frac{1}{8}\). It is therefore enough to construct the interpolation on
\[
\mathcal C_\delta:=S_1^+\times(0,\delta).
\]

\medskip
\noindent
\emph{Cell decomposition.}
Decompose \(\mathcal C_\delta\) into cells of diameter comparable to
\(\delta\). Prescribe
\[
z(\omega,0)=v(\omega),
\quad
z(\omega,\delta)=w(\omega),
\]
and set
\[
z=q
\quad\text{on }E\times(0,\delta).
\]
Choose a translated grid for which the traces on the one-skeleton satisfy
\[
\begin{aligned}
\mathcal E_1
\leq{}&
C\delta
\int_{S_1^+}
\bigl(
|\nabla_{\tan}v|^2+
|\nabla_{\tan}w|^2
\bigr)\,\mathrm{d} S\\
&+
C\delta^{-1}
\int_{S_1^+}| v-w|^2\,\mathrm{d} S.
\end{aligned}
\]
Such a grid exists by the averaging argument in the Luckhaus construction.

\medskip
\noindent
\emph{Extension over the cells.}
Interpolate along the edges by minimizing geodesics in \(\mathbb S^2\).
Since
\[
\pi_1(\mathbb S^2)=0,
\]
the resulting maps on the boundaries of the two-dimensional cells are
null-homotopic. More is needed than this topological fact: after rescaling
each cell to unit size, the quantitative extension step in the Luckhaus
construction~\cite{Luckhaus1988} gives an \(H^1\) extension whose Dirichlet
energy is bounded by a universal multiple of the one-skeleton energy.
Scaling back and summing over the cells preserves the bound
\(C\mathcal E_1\). Extend each two-dimensional trace radially into the
corresponding three-dimensional cell. The possible point singularity at the
center of a cell has finite \(H^1\) energy in dimension three. Consequently,
\[
\int_{\mathcal C_\delta}|\nabla z|^2
\leq C\mathcal E_1,
\]
which gives \eqref{eq:luckhaus-estimate} after composition with \(\Phi^{-1}\).

For every cell meeting \(E\times(0,\delta)\), the construction is performed
relative to the constant subcomplex on which \(z=q\). Hence the flat trace
remains exactly \(q\).

\medskip
\noindent
\emph{Unequal flat traces.}
If the traces are \(g_v\) and \(g_w\), apply the same cell construction on
\(E\times(0,\delta)\). Its normal and tangential contributions are bounded by
\[
C\delta^{-1}\| g_v-g_w\|_{L^2(E)}^2
\]
and
\[
C\delta
\left(
\|\nabla_{\tan}g_v\|_{L^2(E)}^2
+
\|\nabla_{\tan}g_w\|_{L^2(E)}^2
\right),
\]
respectively. The remaining cells are treated relative to this flat
interpolation.

Finally, a uniformly bi-Lipschitz flattening transfers the construction to
\(C^1\) half-balls and preserves the estimate up to a uniform multiplicative
constant.
\end{proof}

We shall also use the following filling consequence. It follows from the same cell construction, now coning the whole boundary after choosing the grid size optimally.

\begin{lemma}[Half-ball filling estimate]
\label{lem:filling}
Let
\[
\Sigma_r:=\partial B_r^+=S_r^+\cup\Gamma_r,
\]
and suppose that
\[
h\in H^1(\Sigma_r;\mathbb S^2).
\]
Then there exists
\[
H\in H^1(B_r^+;\mathbb S^2),
\quad
\operatorname{Tr}_{\Sigma_r}H=h,
\]
such that, for every \(\xi\in\mathbb R^3\),
\begin{equation}
\int_{B_r^+}|\nabla H|^2\,\mathrm{d} x
\leq
C
\left(
\int_{\Sigma_r}|\nabla_{\tan}h|^2\,\mathrm{d} S
\right)^{\frac{1}{2}}
\left(
\int_{\Sigma_r}| h-\xi|^2\,\mathrm{d} S
\right)^{\frac{1}{2}}.
\label{eq:halfball-filling}
\end{equation}
The constant \(C\) is independent of \(r\), \(h\), and \(\xi\).

If \(h=q\in\mathbb S^2\) on \(\Gamma_r\), then
\begin{equation}
\int_{B_r^+}|\nabla H|^2\,\mathrm{d} x
\leq
Cr
\left(
\int_{S_r^+}|\nabla_{\tan}h|^2\,\mathrm{d} S
\right)^{\frac{1}{2}}.
\label{eq:constant-trace-filling}
\end{equation}
\end{lemma}

\begin{proof}
By scaling, it suffices to consider \(r=1\). Let
\[
v\in H^1(B_1^+;\mathbb R^3)
\]
be the componentwise harmonic extension of \(h\). The trace estimate for the
Dirichlet problem gives
\[
\int_{B_1^+}|\nabla v|^2\,\mathrm{d} x
\leq
C\lvert h\rvert_{\dot H^{\frac{1}{2}}(\Sigma_1)}^2.
\]
Interpolation on the compact Lipschitz surface \(\Sigma_1\) yields
\[
\lvert h\rvert_{\dot H^{\frac{1}{2}}(\Sigma_1)}^2
\leq
C
\| h-\xi\|_{L^2(\Sigma_1)}
\|\nabla_{\tan}h\|_{L^2(\Sigma_1)}
\]
for every constant \(\xi\in\mathbb R^3\). Hence
\begin{equation}
\int_{B_1^+}|\nabla v|^2\,\mathrm{d} x
\leq
C
\| h-\xi\|_{L^2(\Sigma_1)}
\|\nabla_{\tan}h\|_{L^2(\Sigma_1)}.
\label{eq:harmonic-filling-bound}
\end{equation}

It remains to restore the sphere constraint without changing the trace. For
\(a\in B_{\frac{1}{2}}(0)\), define
\[
\pi_a(y):=\frac{y-a}{| y-a|},
\quad
\phi_a:=\pi_a|_{\mathbb S^2}.
\]
The map \(\phi_a:\mathbb S^2\to\mathbb S^2\) is a diffeomorphism, and
\[
\sup_{| a|\leq\frac{1}{2}}
\left(
\| D\phi_a\|_{L^\infty(\mathbb S^2)}
+
\| D\phi_a^{-1}\|_{L^\infty(\mathbb S^2)}
\right)
\leq C.
\]
Moreover,
\[
| D\pi_a(y)|\leq\frac{C}{| y-a|}.
\]
Therefore, by Fubini's theorem,
\[
\begin{aligned}
\int_{B_{\frac{1}{2}}(0)}
\int_{B_1^+}| D(\pi_a\circ v)|^2\,\mathrm{d} x\,\mathrm{d} a
&\leq
C\int_{B_1^+}|\nabla v(x)|^2
\left(
\int_{B_{\frac{1}{2}}(0)}
\frac{\mathrm{d} a}{| v(x)-a|^2}
\right)\,\mathrm{d} x\\
&\leq
C\int_{B_1^+}|\nabla v|^2\,\mathrm{d} x.
\end{aligned}
\]
Consequently, for some \(a\in B_{\frac{1}{2}}(0)\),
\[
\int_{B_1^+}| D(\pi_a\circ v)|^2\,\mathrm{d} x
\leq
C\int_{B_1^+}|\nabla v|^2\,\mathrm{d} x.
\]

Define
\[
H:=\phi_a^{-1}\circ\pi_a\circ v.
\]
Then \(H\in H^1(B_1^+;\mathbb S^2)\). Since \(v=h\in\mathbb S^2\) on
\(\Sigma_1\),
\[
\operatorname{Tr}_{\Sigma_1}H
=
\phi_a^{-1}\circ\pi_a(h)
=
\phi_a^{-1}\circ\phi_a(h)
=
h.
\]
The uniform Lipschitz bound for \(\phi_a^{-1}\) gives
\[
\int_{B_1^+}|\nabla H|^2\,\mathrm{d} x
\leq
C\int_{B_1^+}|\nabla v|^2\,\mathrm{d} x.
\]
Combining this estimate with \eqref{eq:harmonic-filling-bound} proves
\eqref{eq:halfball-filling} for \(r=1\). Rescaling gives the result for
arbitrary \(r>0\).

Finally, suppose that \(h=q\) on \(\Gamma_r\). Taking \(\xi=q\), we have
\[
\nabla_{\tan}h=0
\quad\text{on }\Gamma_r
\]
and, since \(h,q\in\mathbb S^2\),
\[
\int_{\Sigma_r}| h-q|^2\,\mathrm{d} S
=
\int_{S_r^+}| h-q|^2\,\mathrm{d} S
\leq
4\lvert S_r^+\rvert
\leq Cr^2.
\]
Substitution into \eqref{eq:halfball-filling} gives
\eqref{eq:constant-trace-filling}.
\end{proof}

\subsection{Strong compactness up to the flat face}

\begin{proposition}[Boundary strong compactness]\label{prop:compactness}
Under the hypotheses of Lemma~\ref{lem:package}, after a subsequence there is $u_0\in H^1(B_2^+;\Sph^2)$ such that, for any $s<2$,
\begin{equation}
u_j\to u_0\quad\text{strongly in }H^1(B_s^+),\quad \Tr_{\Gamma_s}u_0=q,
\end{equation}
and $u_0$ is locally minimizing for $W_0$. Consequently, if $Q_j(x,z,P)$ is any quadratic expression in $P$ whose coefficients converge uniformly on compact sets, then
\begin{equation}
Q_j(x,u_j,\nabla u_j)\to Q_0(x,u_0,\nabla u_0)\quad\text{in }L^1(B_s^+).
\end{equation}
\end{proposition}

\begin{proof}
\emph{Weak compactness.}
After passing to a subsequence,
\[
u_j\rightharpoonup u_0 \quad\text{in }H^1(B_2^+),
\quad
u_j\to u_0 \quad\text{in }L^2(B_2^+)\text{ and a.e.}
\]
The trace theorem and \(g_j\to q\) in \(C^2\) imply
\[
\operatorname{Tr}_{\Gamma_2}u_0=q.
\]
Choose the subsequence so that
\[
\sum_j\| u_j-u_0\|_{L^2(B_2^+)}^2<+\infty.
\]
Fubini's theorem gives the first limit below for almost every \(s<2\), while
Fatou's lemma gives only a finite \(\liminf\), rather than a uniform bound, for
the surface energies. A standard nested subsequence selection therefore
provides radii \(s_m\uparrow2\) and one diagonal subsequence such that, for any
fixed \(m\),
\begin{equation}
\int_{S_{s_m}^+}| u_j-u_0|^2\,\mathrm{d} S\to0,
\quad
\sup_j\int_{S_{s_m}^+}
\bigl(|\nabla_{\tan}u_j|^2+|\nabla_{\tan}u_0|^2\bigr)\,\mathrm{d} S<+\infty.
\label{eq:good-radius-traces}
\end{equation}
Indeed, set \(s_0=0\). At the \(m\)-th stage choose a good radius in
\((\max\{s_{m-1},2-m^{-1}\},2)\), pass to a subsequence realizing the finite
\(\liminf\), and then take the diagonal subsequence.

\medskip
\noindent
\emph{Passage of minimality to the limit.}
Fix one of the radii \(s=s_m\) satisfying
\eqref{eq:good-radius-traces}. Let
\(v\in H^1(B_s^+;\mathbb S^2)\) be an admissible competitor for \(u_0\), with
\[
v=u_0\quad\text{near }S_s^+,
\quad
\operatorname{Tr}_{\Gamma_s}v=q.
\]
Since \(g_j\to q\) in \(C^2\), there exist maps
\[
R_j:\Gamma_s\to SO(3),
\quad
R_j(x')q=g_j(x'),
\quad
R_j\to I\quad\text{in }C^1.
\]
Extend \(R_j\) normally and define
\[
v_j^*(x',x_3):=R_j(x')v(x',x_3).
\]
Then
\[
\operatorname{Tr}_{\Gamma_s}v_j^*=g_j,
\quad
v_j^*\to v\quad\text{strongly in }H^1(B_s^+).
\]

On \(S_s^+\), \(v=u_0\). Set
\[
d_j:=\int_{S_s^+}| u_j-v_j^*|^2\,\mathrm{d} S,
\quad
\delta_j:=
\max\left\{
d_j^{\frac{1}{2}},\,
\| g_j-q\|_{C^1},\,
j^{-1}
\right\}.
\]
Then
\[
d_j\to0,\quad \delta_j\to0,
\quad
\delta_j^{-1}d_j\leq d_j^{\frac{1}{2}}\to0.
\]

Apply the relative Luckhaus lemma in the half-annulus
\[
A_j^+:=B_s^+\backslash B_{s(1-\delta_j)}^+.
\]
It produces an admissible map \(\widetilde v_j\) satisfying
\[
\widetilde v_j=v_j^*
    \quad\text{in }B_{s(1-\delta_j)}^+,
\quad
\widetilde v_j=u_j
    \quad\text{near }S_s^+,
\quad
\operatorname{Tr}_{\Gamma_s}\widetilde v_j=g_j,
\]
and
\[
E_{W_j}(\widetilde v_j;A_j^+)
\leq
C\delta_j+C\delta_j^{-1}d_j+o(1)
=o(1).
\]
Here \(R_j\) is used only to correct the trace of the competitor; neither
\(u_j\) nor \(W_j\) is transformed.

Local minimality of \(u_j\) gives
\[
E_{W_j}(u_j;B_s^+)
\leq
E_{W_j}(\widetilde v_j;B_s^+).
\]
Consequently,
\[
\begin{aligned}
E_{W_0}(u_0;B_s^+)
&\leq
\liminf_{j\to+\infty}E_{W_j}(u_j;B_s^+)\\
&\leq
\limsup_{j\to+\infty}E_{W_j}(u_j;B_s^+)\\
&\leq
E_{W_0}(v;B_s^+).
\end{aligned}
\]
Thus \(u_0\) locally minimizes \(E_{W_0}\).

Taking \(v=u_0\) in the same construction yields
\[
E_{W_j}(u_j;B_s^+)
\to
E_{W_0}(u_0;B_s^+)
\]
for each selected radius \(s=s_m\). Since \(s_m\uparrow2\), these identities
give local minimality and energy convergence on any prescribed inner
half-ball. The preceding diagonal selection fixes a single subsequence.

\medskip
\noindent
\emph{Strong convergence of the gradients.}
Write
\[
W_j(x,z,P)=\frac{1}{2}A_j(x,z)[P,P].
\]
Uniform ellipticity gives
\[
\lambda\int_{B_s^+}|\nabla u_j-\nabla u_0|^2\,\mathrm{d} x
\leq
\int_{B_s^+}
A_j(x,u_j)
[\nabla u_j-\nabla u_0,\nabla u_j-\nabla u_0]\,\mathrm{d} x.
\]
Expanding the right-hand side,
\[
\begin{aligned}
&\int_{B_s^+}A_j(x,u_j)[\nabla u_j,\nabla u_j]\,\mathrm{d} x\\
&\quad
-2\int_{B_s^+}A_j(x,u_j)[\nabla u_j,\nabla u_0]\,\mathrm{d} x\\
&\quad
+\int_{B_s^+}A_j(x,u_j)[\nabla u_0,\nabla u_0]\,\mathrm{d} x.
\end{aligned}
\]
The first term converges by energy convergence. Moreover,
\[
A_j(x,u_j)\nabla u_0
\to
A_0(u_0)\nabla u_0
\quad\text{strongly in }L^2(B_s^+),
\]
by coefficient convergence, almost-everywhere convergence of \(u_j\), and
dominated convergence. Together with
\(\nabla u_j\rightharpoonup\nabla u_0\), this gives convergence of the cross
term. The last term converges by dominated convergence. Hence
\[
\int_{B_s^+}|\nabla u_j-\nabla u_0|^2\,\mathrm{d} x\to0,
\]
and therefore
\[
u_j\to u_0
\quad\text{strongly in }H^1(B_s^+).
\]

Finally, boundedness of \(A_j\), polarization, and strong \(L^2\)-convergence
of the gradients imply the asserted convergence of all quadratic energy and
stress terms:
\[
A_j(x,u_j)[\nabla u_j,\nabla u_j]
\to
A_0(u_0)[\nabla u_0,\nabla u_0]
\quad\text{strongly in }L^1(B_s^+).
\]
The relative version of the Luckhaus lemma is essential here because the
joining half-annulus meets the flat boundary \(\Gamma_s\).
\end{proof}

\subsection{Boundary small-energy theory and density bounds}

\begin{lemma}[Boundary reverse H\"older estimate]
\label{lem:reverseholder}
There exist \(p>2\) and \(C<+\infty\), depending only on the common
ellipticity and coefficient bounds, such that every local minimizer
\[
u\in H^1(B_1^+;\mathbb S^2)
\]
with constant trace
\[
u=q
\quad\text{on }\Gamma_1:=B_1\cap\{x_3=0\}
\]
satisfies
\begin{equation}
\left(
\dashint_{B_{\frac{1}{2}}^+}|\nabla u|^p\,\mathrm{d} x
\right)^{\frac{1}{p}}
\leq
C
\left(
\dashint_{B_1^+}|\nabla u|^2\,\mathrm{d} x
\right)^{\frac{1}{2}}.
\label{eq:boundary-reverse-holder}
\end{equation}
After translation and scaling, the same estimate holds on every concentric
pair of admissible half-balls. The constants are uniform for the coefficient
families and smoothly flattened domains appearing in
Lemma~\ref{lem:package}.
\end{lemma}

\begin{proof}
Set
\[
f:=|\nabla u|.
\]
We first establish a weak reverse H\"older inequality near the flat boundary.

\medskip
\noindent
\emph{Zero-trace Sobolev--Poincar\'e inequality.}
Let \(a\in\{x_3=0\}\) and assume that \(B_{2r}^+(a)\) is an admissible
half-ball. Extend \(u-q\) by zero to the lower half of \(B_{2r}(a)\).
The extension belongs to \(W^{1,\frac{6}{5}}(B_{2r}(a))\) and vanishes on a
fixed positive fraction of that ball. The relative Sobolev--Poincar\'e
inequality on \(B_{2r}(a)\), followed by restriction to the upper half,
gives
\begin{equation}
\left(
\int_{B_{2r}^+(a)}| u-q|^2\,\mathrm{d} x
\right)^{\frac{1}{2}}
\leq
C
\left(
\int_{B_{2r}^+(a)}f^{\frac{6}{5}}\,\mathrm{d} x
\right)^{\frac{5}{6}}.
\label{eq:boundary-sobolev-poincare}
\end{equation}

\medskip
\noindent
\emph{Choice of radius.}
By the coarea formula, there is \(s\in(r,2r)\) such that
\[
\int_{S_s^+(a)}|\nabla_{\tan}u|^2\,\mathrm{d} S
\leq
\frac{C}{r}
\int_{B_{2r}^+(a)}f^2\,\mathrm{d} x
\]
and
\[
\int_{S_s^+(a)}| u-q|^2\,\mathrm{d} S
\leq
\frac{C}{r}
\int_{B_{2r}^+(a)}| u-q|^2\,\mathrm{d} x,
\]
where
\[
S_s^+(a):=\partial B_s(a)\cap\mathbb R^3_+.
\]

\medskip
\noindent
\emph{Filling comparison.}
Apply Lemma~\ref{lem:filling} to the trace of \(u\) on
\(\partial B_s^+(a)\), taking \(\xi=q\). This gives
\(H\in H^1(B_s^+(a);\mathbb S^2)\) such that
\[
H=u\quad\text{on }\partial B_s^+(a)
\]
and
\[
\int_{B_s^+(a)}|\nabla H|^2\,\mathrm{d} x
\leq
C
\left(
\int_{S_s^+(a)}|\nabla_{\tan}u|^2\,\mathrm{d} S
\right)^{\frac{1}{2}}
\left(
\int_{S_s^+(a)}| u-q|^2\,\mathrm{d} S
\right)^{\frac{1}{2}}.
\]
Minimality of \(u\) and uniform ellipticity imply
\[
\begin{aligned}
\int_{B_r^+(a)}f^2\,\mathrm{d} x
&\leq
\int_{B_s^+(a)}f^2\,\mathrm{d} x                                      \\
&\leq
C\int_{B_s^+(a)}|\nabla H|^2\,\mathrm{d} x                             \\
&\leq
\frac{C}{r}
\left(
\int_{B_{2r}^+(a)}f^2\,\mathrm{d} x
\right)^{\frac{1}{2}}
\left(
\int_{B_{2r}^+(a)}| u-q|^2\,\mathrm{d} x
\right)^{\frac{1}{2}}.
\end{aligned}
\]
Using \eqref{eq:boundary-sobolev-poincare} and Young's inequality, we obtain,
for every \(\eta>0\),
\[
\int_{B_r^+(a)}f^2\,\mathrm{d} x
\leq
\eta
\int_{B_{2r}^+(a)}f^2\,\mathrm{d} x
+
\frac{C_\eta}{r^2}
\left(
\int_{B_{2r}^+(a)}f^{\frac{6}{5}}\,\mathrm{d} x
\right)^{\frac{5}{3}}.
\]
Dividing by \(\lvert B_r^+(a)\rvert\) gives
\begin{equation}
\dashint_{B_r^+(a)}f^2\,\mathrm{d} x
\leq
C\eta
\dashint_{B_{2r}^+(a)}f^2\,\mathrm{d} x
+
C_\eta
\left(
\dashint_{B_{2r}^+(a)}f^{\frac{6}{5}}\,\mathrm{d} x
\right)^{\frac{5}{3}}.
\label{eq:weak-boundary-reverse-holder}
\end{equation}

\medskip
\noindent
\emph{Higher integrability.}
Choose \(\eta>0\) so that the first coefficient in
\eqref{eq:weak-boundary-reverse-holder} is below the absorption threshold.
The weak reverse H\"older lemma and Gehring's self-improvement then yield
some \(p>2\) for which
\[
\left(
\dashint_{B_{\frac{r}{2}}^+(a)}f^p\,\mathrm{d} x
\right)^{\frac{1}{p}}
\leq
C
\left(
\dashint_{B_r^+(a)}f^2\,\mathrm{d} x
\right)^{\frac{1}{2}}.
\]
For balls contained in the open half-space, the same argument uses the
corresponding full-ball filling estimate. A finite covering of \(B_{\frac{1}{2}}^+\)
by boundary half-balls and interior balls proves
\eqref{eq:boundary-reverse-holder}.

All constants used above depend only on the uniform ellipticity, coefficient,
and bi-Lipschitz flattening bounds. Hence the estimate is uniform for the
families in Lemma~\ref{lem:package}. This is the boundary analogue of the
higher-integrability argument of Hardt, Kinderlehrer, and
Lin~\cite[Theorem~4.1]{HKL1988}.
\end{proof}
\begin{proposition}[Boundary energy improvement]\label{prop:improvement}
There are $\varepsilon_0>0$, $\theta\in(0,\frac{1}{4})$, and $C<+\infty$ with the following property. Let $W(z,P)=\frac{1}{2}A^{ab}_{ik}(z)P_i^aP_k^b$ have the ellipticity and $C^2$ bounds used above, and let $u$ minimize $E_W$ in $B_1^+$ with trace $q$ on $\Gamma_1$. If
\begin{equation*}
e(u,1):=\int_{B_1^+}|\nabla u|^2\leq\varepsilon_0,
\end{equation*}
then
\begin{equation}
\theta^{-1}\int_{B_\theta^+}|\nabla u|^2\leq\frac{1}{2}\int_{B_1^+}|\nabla u|^2.
\label{eq:boundary-energy-improvement}
\end{equation}
After scaling, the analogous assertion holds in $B_r^+$. More generally,
after flattening a $C^2$ boundary at scale $r$, suppose that
\[
\omega(r)
:=
\| A_r-A_0\|_{C^1}
+\| g_r-q\|_{C^2}
\leq Cr.
\]
Then the right-hand side of \eqref{eq:boundary-energy-improvement} acquires
the additional term $C\omega(r)^2$. Sufficiently small normalized energy
therefore gives H\"older regularity in $B_{\frac{1}{2}}^+\cup\Gamma_{\frac{1}{2}}$. If the
coefficients, boundary, and boundary data are $C^\infty$, the map is
$C^\infty$ there; if the coefficients are analytic, it is analytic on the
interior regular set.
\end{proposition}

\begin{proof}
We prove the decay by contradiction, including the normalization which is sometimes suppressed in statements of boundary $\varepsilon$-regularity. If it failed, there would be autonomous densities $W_j$ with common structural bounds and minimizers $u_j$ with constant flat traces $q_j$ such that
\[
\epsilon_j^2:=\int_{B_1^+}|\nabla u_j|^2\downarrow0,
\quad
\theta^{-1}\int_{B_\theta^+}|\nabla u_j|^2>\frac{1}{2}\epsilon_j^2.
\]
After a subsequence, $q_j\to q$ and $A_j\to A_0$ in $C^1(\Sph^2)$. The Poincar\'e inequality for functions vanishing on $\Gamma_1$ gives $\| u_j-q_j\|_2\leq C\epsilon_j$. Set
\[
w_j=\frac{u_j-q_j}{\epsilon_j}.
\]
Then $w_j$ is bounded in $H^1$, has zero flat trace, and, after a subsequence, converges weakly in $H^1$ and strongly in $L^2$ to a map $w:B_1^+\to T_q\Sph^2$. Tangency follows from
\[
q_j\cdot w_j=-\frac{\epsilon_j}{2}| w_j|^2.
\]
Lemma~\ref{lem:reverseholder}, applied on compact smaller half-balls, also gives $\|\nabla w_j\|_{L^p}\leq C$ for one $p>2$. Consequently the family $|\nabla w_j|^2$ is equi-integrable. Since $u_j\to q$ in measure, this implies
\begin{equation*}
\int_{B_s^+}| A_j(u_j)-A_0(q)|\,
|\nabla w_j|^2\,\mathrm{d} x\to0
\quad s<1.
\end{equation*}

Divide the minimality inequality by $\epsilon_j^2$. After a further
subsequence, arrange
\[
\sum_j\| w_j-w\|_{L^2(B_1^+)}^2<+\infty.
\]
Fubini's theorem gives spherical $L^2$ convergence for almost every radius,
whereas Fatou's lemma gives a finite $\liminf$ of the spherical tangential
energies. Choose one such $s\in(\frac{3}{4},1)$ and pass to a further subsequence
realizing that $\liminf$. Then
\begin{equation*}
\int_{S_s^+}| w_j-w|^2\to0,
\quad
\sup_j\int_{S_s^+}
\bigl(|\nabla_{\tan}w_j|^2
+|\nabla_{\tan}w|^2\bigr)<+\infty.
\end{equation*}
Let $\varphi\in H^1(B_s^+;T_q\Sph^2)$ be a competitor with zero flat trace and $\varphi=w$ near $S_s^+$. Set
\[
v_j=\frac{q_j+\epsilon_j\varphi_j}
{| q_j+\epsilon_j\varphi_j|},
\]
where $\varphi_j$ is the parallel transport of $\varphi$ from $T_q\Sph^2$ to $T_{q_j}\Sph^2$. On $S_s^+$ put
\[
d_j=\epsilon_j^{-2}\int_{S_s^+}| u_j-v_j|^2.
\]
The preceding trace convergence and the expansion
$\epsilon_j^{-1}(v_j-q_j)=\varphi_j
+O\bigl(\epsilon_j|\varphi_j|^2\bigr)$ gives $d_j\to0$
(use $H^1(S_s^+)\hookrightarrow L^4(S_s^+)$); the tangential energies of
$u_j$ and $v_j$, divided by $\epsilon_j^2$, are uniformly bounded. Apply
Lemma~\ref{lem:luckhaus} in an annulus of relative thickness
$\delta_j=\max\{d_j^{\frac{1}{2}},j^{-1}\}$. Both flat traces equal $q_j$,
and the annular energy divided by $\epsilon_j^2$ is
\[
C(\delta_j+\delta_j^{-1}d_j)=o(1).
\]
Minimality, the coefficient convergence estimate above, and weak lower semicontinuity show that $w$ locally minimizes the frozen quadratic functional
\begin{equation*}
Q_q(\varphi)=\frac{1}{2}\int_{B_1^+}A^{ab}_{0,ik}(q)\partial_i\varphi^a\partial_k\varphi^b
\end{equation*}
in $B_s^+$ among $T_q\Sph^2$-valued maps with its own spherical trace and zero trace on $\Gamma_s$. Smooth compactly supported tangent-valued variations are dense in this affine class, so $w$ satisfies the associated linear Euler system.

Taking $\varphi=w$ in the same splice gives convergence of the normalized energies on any smaller good radius. Ellipticity, exactly as in Proposition~\ref{prop:compactness}, then yields
\begin{equation*}
w_j\to w\quad\text{strongly in }H^1(B_s^+).
\end{equation*}
The Euler equation of the frozen quadratic functional above is a constant-coefficient strongly elliptic system with zero Dirichlet data on the flat face. Iterated boundary $H^m$ estimates followed by Sobolev embedding give $\|\nabla w\|_{L^\infty(B_{\frac{1}{2}}^+)}\leq C\|\nabla w\|_{L^2(B_{\frac{3}{4}}^+)}$, and hence
\begin{equation*}
\rho^{-1}\int_{B_\rho^+}|\nabla w|^2\leq C\rho^2\int_{B_{\frac{3}{4}}^+}|\nabla w|^2,\quad0<\rho<\frac{1}{2}.
\end{equation*}
Choose $\theta$ so that $C\theta^2<\frac{1}{4}$. Passing to the limit in the failed version of \eqref{eq:boundary-energy-improvement} by the strong convergence above gives a left side at most $\frac{1}{4}$, contradicting the lower bound $\frac{1}{2}$. This proves \eqref{eq:boundary-energy-improvement}.

Iteration gives, for some $\alpha>0$,
\[
\int_{B_\rho^+}|\nabla u|^2\leq C\rho^{1+2\alpha}\int_{B_1^+}|\nabla u|^2.
\]
The boundary Morrey lemma makes $u$ H\"older continuous. Its image in a
smaller half-ball lies in one coordinate chart of $\Sph^2$. There it is a
weak solution of a uniformly elliptic system with quadratic natural growth.
The difference-quotient estimate, using the displayed Morrey decay to absorb
the quadratic term, yields baseline boundary regularity under the stated
$C^2$ assumptions. If the coefficients and boundary data are smooth,
iteration gives $C^\infty$ regularity; analytic coefficients give interior
analyticity. This is the bootstrap in the small-energy lemma of Hardt,
Kinderlehrer, and Lin~\cite[Lemma~2.5]{HKL1986}; it uses only the small-energy decay just proved,
not global boundary regularity.

We finally prove the perturbative assertion for a flattened boundary and a
non-constant trace. Decrease \(\theta\) if necessary. We claim that there are
\(\varepsilon_0>0\) and \(C_0<+\infty\) such that
\begin{equation}
\theta^{-1}\int_{B_\theta^+}|\nabla u|^2\,\mathrm{d} x
\leq
\frac{1}{2}\int_{B_1^+}|\nabla u|^2\,\mathrm{d} x
+C_0\omega(r)^2.
\label{eq:perturbed-boundary-improvement}
\end{equation}
Suppose otherwise. There are \(r_j\downarrow0\), flattened minimizers
\(u_j\), coefficient fields \(A_j:=A_{r_j}\), and flat traces
\(g_j:=g_{r_j}\) for which
\[
e_j:=\int_{B_1^+}|\nabla u_j|^2\,\mathrm{d} x,
\quad
\omega_j:=\omega(r_j),
\quad
e_j+\omega_j^2\to0,
\]
while \eqref{eq:perturbed-boundary-improvement} fails. Set
\[
q_j:=g_j(0),
\quad
\delta_j:=\bigl(e_j+\omega_j^2\bigr)^{\frac{1}{2}}.
\]

Because \(g_j\) is \(C^2\)-close to \(q_j\), a smooth local section of
\(SO(3)\to\mathbb S^2\) gives
\[
R_j:\Gamma_1\to SO(3),
\quad
R_j(x')q_j=g_j(x'),
\quad
\| R_j-I\|_{C^2(\Gamma_1)}\leq C\omega_j.
\]
Extend \(R_j\) normally and put \(\widetilde u_j=R_j^{-1}u_j\). Its flat
trace is \(q_j\). Let \(\widehat E_j\) denote the quadratic principal part
obtained from the transformed energy by retaining
\(R_j\nabla\widetilde u_j\) and omitting
\((\nabla R_j)\widetilde u_j\). It has the same ellipticity constants.
Expanding the actual transformed energy and applying Young's inequality shows
that, for every admissible comparison region \(G\) and every fixed
\(\eta>0\),
\begin{equation}
\begin{aligned}
\widehat E_j(\widetilde u_j;G)
\leq{}&\widehat E_j(\widetilde v;G)
+\eta\int_G
\bigl(
|\nabla\widetilde u_j|^2
+|\nabla\widetilde v|^2
\bigr)\,\mathrm{d} x
+C_\eta\omega_j^2\lvert G\rvert.
\end{aligned}
\label{eq:rotated-quasiminimality}
\end{equation}
Fix \(\eta\) below the Caccioppoli absorption threshold; the term containing
the energy of \(\widetilde u_j\) is then absorbed. The filling proof of
Lemma~\ref{lem:reverseholder}, now applied to
\eqref{eq:rotated-quasiminimality}, gives some \(p>2\) and, for every
\(s<1\),
\begin{equation}
\|\nabla\widetilde u_j\|_{L^p(B_s^+)}
\leq
C_s\bigl(e_j^{\frac{1}{2}}+\omega_j\bigr)
\leq C_s\delta_j.
\label{eq:normalized-perturbed-reverse-holder}
\end{equation}
Since
\[
\nabla u_j
=
R_j\nabla\widetilde u_j+(\nabla R_j)\widetilde u_j,
\quad
|\widetilde u_j|=1,
\]
the same estimate holds with \(\nabla u_j\) in place of
\(\nabla\widetilde u_j\).

Define
\[
w_j:=\frac{u_j-q_j}{\delta_j},
\quad
h_j:=\frac{g_j-q_j}{\delta_j}.
\]
The trace-extension and Poincar\'e inequalities yield
\[
\| w_j\|_{H^1(B_1^+)}\leq C,
\quad
\| h_j\|_{C^2(\Gamma_1)}
\leq C\frac{\omega_j}{\delta_j}\leq C.
\]
Fix \(\alpha\in(0,1)\). After a subsequence, the compact embedding
\(C^2\hookrightarrow C^{1,\alpha}\) gives
\[
q_j\to q,
\quad
w_j\rightharpoonup w\ \text{in }H^1,
\quad
w_j\to w\ \text{in }L^2,
\quad
h_j\to h\ \text{in }C^{1,\alpha}.
\]
No affine form of \(h\) is asserted. From
\[
q_j\cdot w_j=-\frac{\delta_j}{2}| w_j|^2
\]
we obtain
\[
w:B_1^+\to T_q\mathbb S^2,
\quad
\operatorname{Tr}_{\Gamma_1}w=h.
\]
Estimate \eqref{eq:normalized-perturbed-reverse-holder} makes
\(|\nabla w_j|^2\) equiintegrable. Hence, for every \(s<1\),
\begin{equation}
\int_{B_s^+}
| A_j(x,u_j)-A_0(q)|
|\nabla w_j|^2\,\mathrm{d} x
\to0.
\label{eq:perturbed-coefficient-freezing}
\end{equation}

Choose a good radius \(s\in(\frac{3}{4},1)\) and a subsequence such that
\[
\int_{S_s^+}| w_j-w|^2\,\mathrm{d} S\to0,
\quad
\sup_j\int_{S_s^+}
\bigl(
|\nabla_{\tan}w_j|^2
+|\nabla_{\tan}w|^2
\bigr)\,\mathrm{d} S<+\infty.
\]
Given a tangent-valued competitor \(\varphi\) with flat trace \(h\) that
agrees with \(w\) near \(S_s^+\), choose trace extensions
\(\varphi_j\to\varphi\) strongly in \(H^1(B_s^+)\), with flat trace \(h_j\),
and with the same convergence on \(S_s^+\). An \(H^1\)-small truncation,
which is the identity on the uniformly bounded flat traces, allows us to
assume
\[
\delta_j\|\varphi_j\|_{L^\infty(B_s^+)}\leq\frac{1}{2}.
\]
Thus the denominator in the following projection is bounded away from zero:
\[
v_j:=
\frac{q_j+\delta_j\varphi_j}
{| q_j+\delta_j\varphi_j|}.
\]
Since \(q_j+\delta_jh_j=g_j\in\mathbb S^2\), the flat trace of \(v_j\) is
exactly \(g_j\). The good-radius convergence and the fact that
\(\varphi=w\) near \(S_s^+\) give
\[
d_j
:=
\delta_j^{-2}\int_{S_s^+}| u_j-v_j|^2\,\mathrm{d} S
\to0,
\]
while the tangential energies of \(u_j\) and \(v_j\), divided by
\(\delta_j^2\), are uniformly bounded. Apply Lemma~\ref{lem:luckhaus} with
relative thickness
\[
\tau_j:=\max\{d_j^{\frac{1}{2}},j^{-1}\}.
\]
The normalized annular cost is bounded by
\[
C\bigl(\tau_j+\tau_j^{-1}d_j\bigr)=o(1).
\]
Minimality and \eqref{eq:perturbed-coefficient-freezing} show that \(w\) is
a local minimizer of
\[
Q_q(\psi)
:=
\frac{1}{2}\int_{B_s^+}
A_{0,ik}^{ab}(q)\partial_i\psi^a\partial_k\psi^b\,\mathrm{d} x
\]
among tangent-valued maps with flat trace \(h\) and its own spherical trace.
Using \(\varphi=w\) in the same recovery construction gives
\[
w_j\to w
\quad\text{strongly in }H^1(B_{s'}^+)
\]
for every \(s'<s\). Thus \(w\) solves the corresponding
constant-coefficient linear system.

Extend \(h\) from \(\Gamma_s\) with controlled \(C^{1,\alpha}\) norm and
let \(\ell\) solve the constant-coefficient boundary problem
\[
\partial_i\bigl(A_{0,ik}^{ab}(q)\partial_k\ell^b\bigr)=0,
\quad
\operatorname{Tr}_{\Gamma_s}\ell=h.
\]
The boundary Schauder estimate gives
\[
\|\nabla\ell\|_{L^\infty(B_{\frac{1}{2}}^+)}
\leq C\| h\|_{C^{1,\alpha}(\Gamma_s)}.
\]
Since \(w-\ell\) has zero flat trace, boundary estimates give
\begin{equation}
\rho^{-1}\int_{B_\rho^+}|\nabla w|^2\,\mathrm{d} x
\leq
C\rho^2
\left(
\int_{B_{\frac{3}{4}}^+}|\nabla w|^2\,\mathrm{d} x
+\| h\|_{C^{1,\alpha}(\Gamma_s)}^2
\right)
\label{eq:general-trace-linear-decay}
\end{equation}
for \(0<\rho<\frac{1}{2}\).

Put
\[
a_j:=\frac{e_j}{\delta_j^2},
\quad
b_j:=\frac{\omega_j^2}{\delta_j^2},
\quad
a_j+b_j=1.
\]
After a subsequence, \(a_j\to a\) and \(b_j\to b\). Strong convergence and
\(\| h_j\|_{C^2}\leq C\sqrt{b_j}\) give
\[
\int_{B_{\frac{3}{4}}^+}|\nabla w|^2\,\mathrm{d} x\leq a,
\quad
\| h\|_{C^{1,\alpha}(\Gamma_s)}^2\leq Cb.
\]
Choose \(\theta\) so that \eqref{eq:general-trace-linear-decay} at
\(\rho=\theta\) is at most \(\frac{1}{4}(a+b)\). Then
\[
\lim_{j\to+\infty}
\frac{1}{\theta\delta_j^2}\int_{B_\theta^+}|\nabla u_j|^2\,\mathrm{d} x
\leq\frac{1}{4}.
\]
The failure of \eqref{eq:perturbed-boundary-improvement}, with
\(C_0\geq\frac{1}{2}\), instead gives the lower bound
\[
\frac{1}{\theta\delta_j^2}\int_{B_\theta^+}|\nabla u_j|^2
>
\frac{1}{2}a_j+C_0b_j
\geq\frac{1}{2},
\]
a contradiction. Thus, \eqref{eq:perturbed-boundary-improvement} holds. Since
\(\omega(r)\leq Cr\), in physical variables, it reads
\[
\Theta_u(a,\theta r)
\leq
\frac{1}{2}\Theta_u(a,r)+Cr^2.
\]
The error is summable on dyadic scales.
\end{proof}

\begin{lemma}[Universal upper density near a flat boundary]\label{lem:upperdensity}
Let $u$ minimize a uniformly elliptic quadratic density in a half-ball and have constant trace on the flat face. For all sufficiently small $r$,
\begin{equation}
r^{-1}\int_{B_r^+}|\nabla u|^2\leq M,
\label{eq:universal-upper-density}
\end{equation}
where $M$ depends only on the structural constants. With smooth boundary and boundary data, $M$ also depends on their local $C^2$ bounds but is uniform along any blow-up sequence.
\end{lemma}

\begin{proof}
Put 
\[
e(r)=r^{-1}\int_{B_r^+}|\nabla u|^2.
\]
Coarea and
averaging give a radius $s\in(\frac{r}{2},r)$ such that
\[
\int_{S_s^+}|\nabla_{\tan}u|^2\leq Ce(r).
\]
Use Lemma~\ref{lem:filling} to fill the trace of $u$ on $\partial B_s^+$ and use this filling as a competitor. Ellipticity gives
\begin{equation*}
e\left(\frac{r}{2}\right)\leq C_0\sqrt{e(r)}.
\end{equation*}
If $e(r)\geq4C_0^2$, the right side is at most $\frac{1}{2}e(r)$; otherwise, the next iterate is at most $2C_0^2$. Iterating on dyadic radii therefore puts $e$ below a universal bound after finitely many steps. If $\frac{r}{2}<s<r$, monotonicity of the un-normalised energy gives $e(s)\leq2e(r)$, so the dyadic estimate implies \eqref{eq:universal-upper-density} at any smaller radius. Smooth flattening changes the ellipticity constants by $1+O(r)$. If the flat trace is smooth $g$ rather than a constant, its contribution to the first factor in \eqref{eq:halfball-filling} is at most $Cr^2\|\nabla_{\tan}g\|_{L^\infty}^2$, while the second factor is at most $Cr^2$. Thus, the precise perturbed recurrence is
\[
e\left(\frac{r}{2}\right)\leq C_0\sqrt{e(r)}+Cr\| g\|_{C^1}+Cr.
\]
The last two terms are summable on dyadic radii, so the same iteration proves the conclusion uniformly after rescaling.
\end{proof}

\subsection{Consequences for the singular set}
\label{subsec:singular-set}

The reverse H\"older argument of Hardt, Kinderlehrer, and
Lin~\cite[Theorem~4.1]{HKL1988}, together with
Lemma~\ref{lem:reverseholder}, gives
\[
\nabla u\in L^p_{\mathrm{loc}}
\quad\text{for some }p>2.
\]
Fix
\[
2<p_0<\min\{p,3\}.
\]
Then \(\nabla u\in L^{p_0}_{\mathrm{loc}}\). We combine this higher
integrability with the clearing-out criterion.

Let \(d\) denote the domain dimension. At every interior singular point \(a\),
the contrapositive of the full-ball interior analogue of
Proposition~\ref{prop:improvement} gives
\[
\int_{B_r(a)}|\nabla u|^2\,\mathrm{d} x
\geq
\varepsilon_* r^{d-2}
\]
for all sufficiently small \(r\). On the other hand, H\"older's inequality
gives
\[
\int_{B_r(a)}|\nabla u|^2\,\mathrm{d} x
\leq
Cr^{d(1-\frac{2}{p_0})}
\left(
\int_{B_r(a)}|\nabla u|^{p_0}\,\mathrm{d} x
\right)^{\frac{2}{p_0}}.
\]
Combining the two estimates yields
\[
r^{d-p_0}
\leq
C\int_{B_r(a)}|\nabla u|^{p_0}\,\mathrm{d} x.
\]

At every singular point the averages of \(|\nabla u|^2\) diverge
along small scales, so the Lebesgue differentiation theorem shows that the
singular set has \(d\)-dimensional measure zero. A Vitali covering by balls
of radii below \(\delta\), followed by summation over a pairwise disjoint
subfamily, gives on every compact \(K\subset\subset\Omega\)
\[
\mathcal H^{d-p_0}_{C\delta}
\bigl(\operatorname{Sing}_{\mathrm{int}}u\cap K\bigr)
\leq
C\int_{N_{C\delta}(\operatorname{Sing}_{\mathrm{int}}u)\cap K'}
|\nabla u|^{p_0}\,\mathrm{d} x,
\]
where \(K'\subset\subset\Omega\) is a fixed neighbourhood of \(K\). The
right-hand side tends to zero by absolute continuity. Letting
\(\delta\downarrow0\) and using the corollary of Hardt, Kinderlehrer, and
Lin~\cite[Corollary~4.2]{HKL1988} gives, in dimension three,
\[
\dim_{\mathcal H}\operatorname{Sing}_{\mathrm{int}}u
\leq3-p_0<1.
\]

We also record the weaker boundary consequence needed later. In a fixed
flattened boundary disk \(D\), set
\[
\nu:=|\nabla u|^2\,\mathrm{d} x,
\quad
E_m
:=
\left\{
a\in D:
\limsup_{r\downarrow0}
\frac{\nu(B_r(a)\cap\Omega)}{r}
>
\frac{1}{m}
\right\}.
\]
For every \(a\in E_m\), one can choose arbitrarily small radii satisfying
\[
r<m\,\nu(B_r(a)\cap\Omega).
\]
A Vitali selection and the disjointness of the selected balls imply, on every
smaller disk \(D'\subset\subset D\),
\[
\mathcal H^1(E_m\cap D')
\leq
Cm\,\nu(U)<+\infty,
\]
where \(U\) is a fixed boundary neighbourhood. Hence
\[
\mathcal H^2\left(\bigcup_{m=1}^{+\infty}E_m\right)=0.
\]
Since \(D'\) has positive two-dimensional measure, it contains a point
\(a\notin\bigcup_mE_m\). At this point,
\[
\lim_{r\downarrow0}
\frac{1}{r}
\int_{B_r(a)\cap\Omega}|\nabla u|^2\,\mathrm{d} x
=0.
\]
Proposition~\ref{prop:improvement} makes \(a\) regular, and openness of the
regular set provides a nonempty relatively open smooth boundary patch.
Thus this argument proves the boundary nonemptiness assertion in
Lemma~\ref{lem:package} without assuming the full boundary regularity theorem.

\subsection{Oseen--Frank blow-up at a boundary singularity}

We now apply the compactness package to the original Oseen--Frank
functional. The null-Lagrangian reduction and the structural properties of
the coercive representative required in this step were established in
Subsection~\ref{subsec:of-variational-problem}, so they are not repeated here.

\begin{proposition}[Tangent generated by a boundary singularity]\label{prop:tangent}
Let $n$ be an Oseen--Frank minimizer in a smooth bounded three-dimensional domain with smooth strong anchoring. If $a\in\partial\Omega$ is singular, then there are an autonomous analytic density $W_0$ satisfying \eqref{eq:coercivity}, a point $q\in\Sph^2$, and a class-A minimizer $U:\R^3_+\to\Sph^2$ with trace $q$ such that
\begin{equation}
0<cR\leq E_{W_0}(U;B_R^+)\leq CR\quad\text{for any }R>0.
\label{eq:tangent-linear-growth}
\end{equation}
\end{proposition}

\begin{proof}
Translate and rotate so that $a=0$ and the inward normal is $e_3$. Let $F$ be a smooth boundary-flattening diffeomorphism with $F(0)=0$ and $DF(0)=I$. For $r_j\downarrow0$ set
\[
\Omega_j=r_j^{-1}F^{-1}(\Omega),\quad u_j(y)=n(F(r_jy)).
\]
After change of variables, $u_j$ minimizes a density

\[
W_j(y,z,P)=\frac{1}{2}A_j(y,z)[P,P]
\]

on $\Omega_j$. The matrices $A_j$ have common ellipticity and smoothness
bounds on compact sets and converge in $C^1_{\mathrm{loc}}$ to the
autonomous matrix frozen at the origin. By the structural properties
established in Subsection~\ref{subsec:of-variational-problem}, the limiting
density is analytic and uniformly strictly convex for any positive Frank
triple. The rescaled traces converge in $C^2_{\mathrm{loc}}$ to $q=g(0)$.

Lemma~\ref{lem:upperdensity}, including its flattening and trace errors, gives $C,r_0>0$ such that
\[
\int_{\Omega\cap B_r}|\nabla n|^2\leq Cr
\quad 0<r<r_0.
\]
Since $0$ is singular, Proposition~\ref{prop:improvement} gives the matching lower bound $\int_{\Omega\cap B_r}|\nabla n|^2\geq cr$. For fixed $R$ and large $j$, $r_jR<r_0$, so $u_j$ has uniform two-sided bounds on $B_R^+$. Proposition~\ref{prop:compactness}, smooth domain convergence, and a diagonal subsequence give $U\in H^1_{\mathrm{loc}}(\R^3_+;\Sph^2)$ with trace $q$, strong $H^1$ convergence on compact half-balls, and \eqref{eq:tangent-linear-growth}.

If $V$ differs from $U$ in a compact subset and has the same flat trace, splice $V$ to $u_j$ in a slightly larger half-ball by Lemma~\ref{lem:luckhaus}. Minimality of $u_j$ and passage to the limit show $E_{W_0}(U)\leq E_{W_0}(V)$. Thus $U$ is class-A minimizing.
\end{proof}

\section{The positive boundary stress measure}

\subsection{Energy--momentum tensor and inner stationarity}

We now introduce the device that will replace the boundary monotonicity
formula. It applies to a general class of autonomous quadratic integrands, not
only to the Oseen--Frank density. Set
\[
\mathbb R^3_+
:=
\{x=(x',x_3)\in\mathbb R^3:x_3>0\},
\quad
D_R
:=
\{(x',0)\in\partial\mathbb R^3_+:| x'|<R\}.
\]
Let
\[
W:\mathbb S^2\times\mathbb R^{3\times3}\to[0,+\infty)
\]
be autonomous, of class \(C^2\), quadratic in \(P\), and uniformly elliptic in
the sense of \eqref{eq:coercivity}. In particular,
\[
W(z,tP)=t^2W(z,P)
\quad
(z\in\mathbb S^2,\ P\in\mathbb R^{3\times3},\ t\in\mathbb R).
\]
As before, we write
\[
E_W(u;A):=\int_A W(u,\nabla u)\,\mathrm{d} x.
\]

\begin{definition}[Class-A minimizer]
Let \(q\in\mathbb S^2\). A map
\[
u\in H^1_{\mathrm{loc}}
   (\overline{\mathbb R^3_+};\mathbb S^2),
\quad
\operatorname{Tr}_{\partial\mathbb R^3_+}u=q,
\]
is called a \emph{class-A minimizer} of \(E_W\) in \(\mathbb R^3_+\) if
it minimizes the energy with respect to any boundedly supported admissible
perturbation.

More precisely, let \(G\subset\mathbb R^3\) be bounded and Lipschitz, and
suppose that
\[
G^+:=G\cap\mathbb R^3_+
\]
is a Lipschitz domain. If \(v\in H^1(G^+;\mathbb S^2)\) satisfies
\[
\operatorname{Tr}v=\operatorname{Tr}u
\quad\text{on }\partial G^+\cap\mathbb R^3_+,
\quad
\operatorname{Tr}v=q
\quad\text{on }\partial G^+\cap\partial\mathbb R^3_+,
\]
then
\[
E_W(u;G^+)\leq E_W(v;G^+).
\]
\end{definition}

\begin{remark}
Class-A minimality is the natural notion on the unbounded half-space, where
the total energy may be infinite. It requires exact minimization against any
compactly supported perturbation, without any smallness assumption on the
competitor. A global minimizer on a bounded domain with fixed Dirichlet trace
is class-A minimizing on any bounded comparison region, but the converse
need not hold because class-A minimality imposes no condition at infinity.
\end{remark}

For \(P=(P_i^a)\in\mathbb R^{3\times3}\), we use the convention
\[
P_i^a=\partial_i u^a,
\quad
W_{P_i^a}(z,P)
:=
\frac{\partial W}{\partial P_i^a}(z,P).
\]
Here \(i,j\in\{1,2,3\}\) are spatial indices,
\(a\in\{1,2,3\}\) is a target-component index, and repeated indices are
summed. The energy-momentum tensor associated with \(u\) is
\begin{equation}
T_{ij}[u]
:=
W(u,\nabla u)\delta_{ij}
-
W_{P_i^a}(u,\nabla u)\,\partial_j u^a .
\end{equation}

The next proposition derives the inner-variation formula directly at
\(H^1\) regularity and then applies it to class-A minimizers.

\begin{proposition}[\(H^1\) inner variation and inner stationarity]
\label{prop:inner-stationarity}
Let \(K\subset\subset\mathbb R^3\) be a bounded Lipschitz domain,
\(v\in H^1(K;\mathbb S^2)\), and
\(X\in C_c^1(K;\mathbb R^3)\). Define
\[
\Phi_s:=\operatorname{id}-sX
\]
and, for sufficiently small \(\lvert s\rvert\),
\[
J(s)
:=
\int_K
W\bigl(v\circ\Phi_s,\nabla(v\circ\Phi_s)\bigr)\,\mathrm{d} x.
\]
Then \(J\) is differentiable at \(s=0\), and
\begin{equation}
J'(0)
=
\int_K
\left[
W(v,\nabla v)\delta_{ij}
-
W_{P_i^a}(v,\nabla v)\partial_jv^a
\right]
\partial_iX_j\,\mathrm{d} x.
\label{eq:inner-variation}
\end{equation}

Consequently, if \(u\) is a class-A minimizer in \(\mathbb R^3_+\), then
\begin{equation}
\int_{\mathbb R^3_+}
T_{ij}[u]\partial_iX_j\,\mathrm{d} x=0
\quad
\text{for any }
X\in C_c^1(\mathbb R^3_+;\mathbb R^3).
\label{eq:weak-inner-stationarity}
\end{equation}
Equivalently,
\begin{equation}
\partial_iT_{ij}[u]=0
\quad\text{in }\mathcal D'(\mathbb R^3_+).
\label{eq:stress-conservation}
\end{equation}
Moreover,
\[
| T[u]|\leq C|\nabla u|^2
\quad\text{almost everywhere in }\mathbb R^3_+,
\]
where \(C\) depends only on the quadratic-growth bounds for \(W\).
\end{proposition}

\begin{proof}
For sufficiently small \(\lvert s\rvert\), the map
\(\Phi_s=\operatorname{id}-sX\) is bi-Lipschitz and equals the identity
near \(\partial K\). Let
\[
\Psi_s:=\Phi_s^{-1}.
\]
Using the change of variables \(y=\Phi_s(x)\), we obtain
\[
J(s)
=
\int_K
W\left(
v(y),
\nabla v(y)
\bigl[I-sDX(\Psi_s(y))\bigr]
\right)
\det D\Psi_s(y)\,\mathrm{d} y.
\]
This expression involves only \(v\) and its first weak derivatives; in
particular, no second derivative of \(v\) is required.

As \(s\to0\), uniformly on \(K\),
\[
D\Psi_s(y)=I+sDX(y)+o(s)
\]
and hence
\[
\det D\Psi_s(y)
=
1+s\operatorname{div}X(y)+o(s).
\]
Furthermore, in components,
\[
\left[
\nabla v(y)
\bigl(I-sDX(\Psi_s(y))\bigr)
\right]_i^a
=
\partial_i v^a
-
s\,\partial_jv^a\,\partial_iX_j
+
o(s).
\]
Differentiating the integrand at \(s=0\) therefore gives
\[
W(v,\nabla v)\operatorname{div}X
-
W_{P_i^a}(v,\nabla v)
\,\partial_jv^a\,\partial_iX_j.
\]
Since
\[
\operatorname{div}X=\delta_{ij}\partial_iX_j,
\]
this is precisely the integrand in \eqref{eq:inner-variation}.

Because \(W\) is quadratic in \(P\),
\[
\lvert W(z,P)\rvert\leq C| P|^2,
\quad
| W_P(z,P)|\leq C| P|.
\]
The difference quotient is consequently bounded by
\[
C|\nabla v|^2\| DX\|_{L^\infty(K)},
\]
which is integrable. Dominated convergence proves
\eqref{eq:inner-variation} for any \(v\in H^1(K;\mathbb S^2)\).

Now let \(u\) be a class-A minimizer and take
\[
X\in C_c^1(\mathbb R^3_+;\mathbb R^3).
\]
Choose \(K\subset\subset\mathbb R^3_+\) containing \(\operatorname{spt}X\). The maps
\[
u_s:=u\circ\Phi_s
\]
coincide with \(u\) near \(\partial K\) and are therefore admissible
competitors for both positive and negative \(s\). Class-A minimality gives
\[
J(0)\leq J(s)
\]
for all sufficiently small \(\lvert s\rvert\). Thus \(s=0\) is a local minimum of \(J\),
and
\[
J'(0)=0.
\]
Substituting \eqref{eq:inner-variation} yields
\eqref{eq:weak-inner-stationarity}.

By the definition of distributional divergence,
\[
\left\langle \partial_iT_{ij}[u],X_j\right\rangle
=
-\int_{\mathbb R^3_+}
T_{ij}[u]\partial_iX_j\,\mathrm{d} x,
\]
so \eqref{eq:weak-inner-stationarity} is equivalent to
\eqref{eq:stress-conservation}. Finally, the quadratic bounds for \(W\) and
\(W_P\) give
\[
| T[u]|
\leq
C\bigl(|\nabla u|^2
+| W_P(u,\nabla u)|\,|\nabla u|\bigr)
\leq C|\nabla u|^2.
\]
Then we complete the proof.
\end{proof}

\begin{remark}
The vector fields in Proposition~\ref{prop:inner-stationarity} are strictly supported
 inside \(\mathbb R^3_+\), so both signs of the variation parameter
are admissible, and the first variation vanishes. A normal variation reaching
the Dirichlet plane is generally admissible in only one direction. The
resulting one-sided inequality, rather than an equality, is the source of the
positive boundary stress measure constructed below.
\end{remark}

\subsection{The one-sided Dirichlet stress measure}

We associate with the Dirichlet plane a non-negative measure representing the
normal flux of the energy-momentum tensor.

\begin{proposition}[One-sided Dirichlet stress measure]
\label{prop:stress}
Let \(u:\mathbb R^3_+\to\mathbb S^2\) be a class-A minimizer satisfying
\[
\operatorname{Tr}_{\partial\mathbb R^3_+}u=q
\]
for some \(q\in\mathbb S^2\). Then there exists a unique non-negative Radon
measure \(\mu_u\) on
\(\partial\mathbb R^3_+\simeq\mathbb R^2\) such that
\begin{equation}
\int_{\mathbb R^2}\phi\,\mathrm{d} \mu_u
=
\int_{\mathbb R^3_+}
T_{i3}[u]\partial_i\zeta\,\mathrm{d} x
\label{eq:boundary-stress-definition}
\end{equation}
whenever
\[
\phi\in C_c^\infty(\mathbb R^2),
\quad
\zeta\in C_c^\infty(\mathbb R^3),
\quad
\zeta(x',0)=\phi(x').
\]

If \(u\) is \(C^1\) up to an open boundary patch \(D\), then
\begin{equation}
\mu_u\llcorner D
=
W(q,\partial_3u\otimes e_3)\,\mathrm{d} x'.
\label{eq:boundary-stress-density}
\end{equation}
Consequently,
\[
\mu_u\llcorner D=0
\quad\Leftrightarrow\quad
\partial_3u=0\quad\text{on }D.
\]
\end{proposition}

\begin{proof}
We divide the proof into three steps.

\emph{Step 1: independence of the extension.}
Suppose that \(\zeta_1,\zeta_2\in C_c^\infty(\mathbb R^3)\) have the same
trace on \(\{x_3=0\}\), and set
\[
h:=\zeta_1-\zeta_2.
\]
Then \(h(x',0)=0\), and hence, near the plane,
\[
\lvert h(x',x_3)\rvert
\leq
\|\partial_3h\|_{L^\infty}\lvert x_3\rvert.
\]
Smoothness and the zero trace give
\[
h(x',x_3)=x_3k(x',x_3)
\]
for a smooth bounded function \(k\) near the plane. Therefore, for
sufficiently small positive and negative \(s\), the maps
\[
\Phi_s(x)=x-sh(x)e_3
\]
are bi-Lipschitz self-maps of \(\mathbb R^3_+\), fix its boundary, and equal
the identity outside a bounded set. Both signs of the corresponding inner
variation are admissible. On a bounded half-space region containing the
support of \(h\), repeat the change-of-variables calculation from
Proposition~\ref{prop:inner-stationarity}. Class-A minimality for both signs
then gives
\[
\int_{\mathbb R^3_+}
T_{i3}[u]\partial_i h\,\mathrm{d} x=0.
\]
It follows that
\[
\int_{\mathbb R^3_+}
T_{i3}[u]\partial_i\zeta_1\,\mathrm{d} x
=
\int_{\mathbb R^3_+}
T_{i3}[u]\partial_i\zeta_2\,\mathrm{d} x.
\]
Thus the right-hand side of \eqref{eq:boundary-stress-definition} depends only
on the boundary trace \(\phi\).

\emph{Step 2: positivity.}
Extend \(u\) by \(q\) to the lower half-space:
\[
\bar u(x)
:=
\begin{cases}
u(x),&x_3>0,\\
q,&x_3\leq0
\end{cases}.
\]
Since \(u=q\) in the trace sense on the plane,
\[
\bar u\in H^1_{\mathrm{loc}}(\mathbb R^3;\mathbb S^2).
\]

Let \(\phi\in C_c^\infty(\mathbb R^2)\) satisfy \(\phi\geq0\), and choose
\(\zeta\in C_c^\infty(\mathbb R^3)\) such that
\[
\zeta\geq0,
\quad
\zeta(x',0)=\phi(x').
\]
For \(s>0\), set
\[
\Phi_s(x)=x-s\zeta(x)e_3,
\quad
u_s
:=
\left.\bar u\circ\Phi_s\right|_{\mathbb R^3_+}.
\]
Because \(\Phi_s\) moves the boundary into the lower half-space, where
\(\bar u=q\), the map \(u_s\) has the same Dirichlet trace as \(u\). It is
therefore an admissible compactly supported competitor. Choose a ball
\(K\subset\subset\mathbb R^3\) whose interior contains
\(\operatorname{spt}\zeta\), so that \(\zeta=0\) near \(\partial K\) and
\(K^+:=K\cap\mathbb R^3_+\) is Lipschitz. Define the full-space variation
energy
\[
J(s)
:=
\int_K
W\bigl(\bar u\circ\Phi_s,\nabla(\bar u\circ\Phi_s)\bigr)\,\mathrm{d} x.
\]
For \(x_3\leq0\) and \(s\geq0\), one has
\(\Phi_s(x)_3=x_3-s\zeta(x)\leq0\). Hence \(\bar u\circ\Phi_s=q\) there,
and its contribution to \(J(s)\) is zero because \(W(q,0)=0\). Thus the
upper-half-space part of \(J(s)\) is exactly the energy of \(u_s\).
Class-A minimality gives
\[
J'(0+)\geq0.
\]
Apply the general \(H^1(K)\) inner-variation formula in
Proposition~\ref{prop:inner-stationarity} to \(\bar u\), with
\(X=\zeta e_3\). The stress vanishes in the lower half-space, and therefore
\[
0\leq J'(0+)=
\int_{\mathbb R^3_+}
T_{i3}[u]\partial_i\zeta\,\mathrm{d} x.
\]
Hence, the functional
\[
\phi\mapsto
\int_{\mathbb R^3_+}
T_{i3}[u]\partial_i\zeta\,\mathrm{d} x
\]
is a positive distribution on \(\mathbb R^2\). Any positive distribution is
a non-negative Radon measure, which defines \(\mu_u\). Its uniqueness follows
from the uniqueness of Radon measures determined by their action on test
functions.

\emph{Step 3: density on a regular boundary patch.}
Assume that \(u\) is \(C^1\) up to an open boundary patch \(D\). Since
\[
\partial_iT_{i3}=0
\quad\text{in }\mathbb R^3_+,
\]
integration by parts and the outward normal
\(\nu=-e_3\) give
\[
\mathrm{d} \mu_u=-T_{33}[u]\,\mathrm{d} x'
\quad\text{on }D.
\]

The constant boundary condition implies
\[
\partial_1u=\partial_2u=0
\quad\text{on }D,
\]
so that
\[
\nabla u=\partial_3u\otimes e_3
\quad\text{on }D.
\]
By the two-homogeneity of \(W\),
\[
W_{P_3^a}
   (q,\partial_3u\otimes e_3)\,\partial_3u^a
=
2W(q,\partial_3u\otimes e_3).
\]
Therefore,
\[
T_{33}[u]
=
W(q,\partial_3u\otimes e_3)
-
W_{P_3^a}
   (q,\partial_3u\otimes e_3)\,\partial_3u^a
=
-W(q,\partial_3u\otimes e_3).
\]
This proves \eqref{eq:boundary-stress-density}. Finally, uniform coercivity
gives
\[
W(q,\partial_3u\otimes e_3)
\geq
\lambda|\partial_3u|^2,
\]
and hence
\[
\mu_u\llcorner D=0
\quad\Leftrightarrow\quad
\partial_3u=0\quad\text{on }D.
\]
\end{proof}

\begin{remark}
The positivity of \(\mu_u\) comes from the direction of the boundary
variation: pushing the Dirichlet plane into the constant extension is
admissible, whereas the opposite normal variation does not need to preserve the
boundary condition.
\end{remark}

\subsection{Flux identity, finiteness, and scaling}

For almost every \(R>0\), define the normal configurational flux through the
upper hemisphere by
\[
F_u(R)
:=
\int_{\partial B_R\cap\mathbb R^3_+}
T_{i3}[u]\nu_i\,\mathrm{d} S
=
\int_{\partial B_R\cap\mathbb R^3_+}
T_{i3}[u]\frac{x_i}{R}\,\mathrm{d} S,
\]
where \(\nu=\frac{x}{R}\) is the outward unit normal to \(\partial B_R\).

\begin{lemma}
\label{lem:flux}
Let \(u:\mathbb R^3_+\to\mathbb S^2\) be a class-A minimizer with constant
trace \(q\), and let \(\mu_u\) be the boundary stress measure from
Proposition~\ref{prop:stress}. Then the following statements hold.

\begin{enumerate}[label=$(\theenumi)$]
\item \emph{Flux identity.} For almost every \(R>0\),
\begin{equation}
F_u(R)=-\mu_u(D_R).
\label{eq:flux-identity}
\end{equation}

\item \emph{Finite total boundary stress.} If
\begin{equation}
E_W(u;B_R^+)\leq MR
\quad\text{for any }R>0,
\label{eq:global-linear-upper-growth}
\end{equation}
then
\[
\mu_u(\mathbb R^2)
\leq C(\lambda,\Lambda)M.
\]

\item \emph{Scaling law.} For \(R>0\), let
\[
u_R(x):=u(Rx).
\]
Then, for any Borel set \(A\subset\mathbb R^2\),
\begin{equation}
\mu_{u_R}(A)=\mu_u(RA).
\label{eq:boundary-stress-scaling}
\end{equation}
\end{enumerate}
\end{lemma}

\begin{proof}
\emph{Step 1: flux identity.}
Choose
\[
\eta\in C_c^\infty([0,+\infty))
\]
constant near \(0\), and use
\[
\zeta(x):=\eta(| x|)
\]
in the definition of \(\mu_u\). Since
\[
\partial_i\zeta(x)
=
\eta'(| x|)\frac{x_i}{| x|},
\]
the coarea formula gives
\begin{equation*}
\int_{\mathbb R^2}\eta(| x'|)\,\mathrm{d} \mu_u(x')
=
\int_0^{+\infty} \eta'(r)F_u(r)\,\mathrm{d} r.
\end{equation*}

Set
\[
m(r):=\mu_u(D_r).
\]
Stieltjes integration by parts yields
\begin{equation*}
\int_{\mathbb R^2}\eta(| x'|)\,\mathrm{d} \mu_u(x')
=
-\int_0^{+\infty}\eta'(r)m(r)\,\mathrm{d} r.
\end{equation*}
This formula also includes a possible atom of \(\mu_u\) at the origin.
Comparing the preceding identities, we obtain
\[
\int_0^{+\infty}
\eta'(r)\bigl(F_u(r)+m(r)\bigr)\,\mathrm{d} r=0.
\]
Hence \(F_u+m\) is constant almost everywhere. Since the value of
\(\eta(0)\) is arbitrary, that constant must be zero. Therefore
\[
F_u(R)=-m(R)=-\mu_u(D_R)
\]
for almost every \(R>0\), proving \eqref{eq:flux-identity}.

\emph{Step 2: finite total mass.}
The quadratic-growth estimate for the energy-momentum tensor gives
\[
\lvert F_u(s)\rvert
\leq
C\int_{\partial B_s\cap\mathbb R^3_+}
|\nabla u|^2\,\mathrm{d} S
\]
for almost every \(s>0\). By coercivity and
\eqref{eq:global-linear-upper-growth},
\[
\int_{B_{2R}^+\backslash B_R^+}|\nabla u|^2\,\mathrm{d} x
\leq
\lambda^{-1}E_W(u;B_{2R}^+)
\leq
\frac{2M}{\lambda}R.
\]
The coarea formula therefore provides a radius \(s\in(R,2R)\), chosen also
so that \eqref{eq:flux-identity} holds, such that
\[
\int_{\partial B_s\cap\mathbb R^3_+}
|\nabla u|^2\,\mathrm{d} S
\leq C(\lambda)M.
\]
Consequently,
\[
\lvert F_u(s)\rvert\leq C(\lambda,\Lambda)M.
\]

Since \(\mu_u\geq0\) and \(D_R\subset D_s\), the flux identity gives
\[
\mu_u(D_R)
\leq
\mu_u(D_s)
=
-F_u(s)
\leq
C(\lambda,\Lambda)M.
\]
Letting \(R\to+\infty\) proves
\[
\mu_u(\mathbb R^2)
\leq
C(\lambda,\Lambda)M.
\]

\emph{Step 3: scaling.}
Because \(W\) is quadratic in its matrix variable,
\[
\nabla u_R(x)=R\nabla u(Rx),
\quad
T[u_R](x)=R^2T[u](Rx).
\]
Let \(\phi\in C_c^\infty(\mathbb R^2)\), and let
\(\zeta\in C_c^\infty(\mathbb R^3)\) satisfy
\[
\zeta(x',0)=\phi(x').
\]
Using \eqref{eq:boundary-stress-definition} and changing variables
\(y=Rx\), we find
\begin{align*}
\int_{\mathbb R^2}\phi(x')\,\mathrm{d} \mu_{u_R}(x')
&=
\int_{\mathbb R^3_+}
R^2T_{i3}[u](Rx)\,\partial_i\zeta(x)\,\mathrm{d} x\\
&=
\int_{\mathbb R^3_+}
T_{i3}[u](y)\,
\partial_{y_i}\!\left[\zeta\!\left(\frac{y}{R}\right)\right]\,\mathrm{d} y\\
&=
\int_{\mathbb R^2}
\phi\!\left(\frac{y'}{R}\right)\,\mathrm{d} \mu_u(y').
\end{align*}
Thus, \(\mu_{u_R}\) is the push-forward of \(\mu_u\) under
\(y'\mapsto \frac{y'}{R}\), which is equivalent to
\[
\mu_{u_R}(A)=\mu_u(RA)
\]
for any Borel set \(A\subset\mathbb R^2\).
\end{proof}

\begin{remark}[Role of the dimension]
The scaling law \eqref{eq:boundary-stress-scaling} does not contain a multiplicative
factor because the domain has dimension three. In dimension \(m\), the
corresponding formula would be
\[
\mu_{u_R}(A)=R^{3-m}\mu_u(RA).
\]
Thus, the exact scale invariance of the boundary stress measure used in the following is
specific to \(m=3\).
\end{remark}

\section{Large-scale rigidity and proof of the main theorem}

This section establishes the Liouville theorem needed to exclude boundary
singularities and then completes the proof of Theorem~\ref{thm:main}. The
argument has three ingredients: analytic Cauchy uniqueness, connectedness of
the regular set, and concentration of the boundary stress under large-scale
rescaling.

\subsection{Analytic Cauchy uniqueness}

\begin{lemma}
\label{lem:analytic}
Let \(v\) be a local minimizer in \(B_1^+\) of an autonomous, analytic,
uniformly elliptic quadratic density. Suppose that \(v\) is smooth up to
\(\Gamma_1\) and that, on a non-empty relatively open disk \(D\subset\Gamma_1\),
\[
v=q,
\quad
\partial_3v=0.
\]
Then
\[
v=q
\]
in an open non-empty subset of \(B_1^+\).
\end{lemma}

\begin{proof}
The proof is divided into several steps as follows.

\vspace{1em}
\textit{Step 1.} Shrink \(D\) so that \(v\) takes values in an analytic chart
\[
\Psi:V\subset\mathbb R^2\to\mathbb S^2,
\quad
\Psi(0)=q.
\]
Write
\[
v=\Psi(y).
\]
In these coordinates the energy density has the form
\[
L(y,Dy)
=
\frac{1}{2}
a_{ij}^{\alpha\beta}(y)
\partial_i y^\alpha\partial_j y^\beta,
\]
where \(a_{ij}^{\alpha\beta}\) is analytic and strongly elliptic.

\vspace{1em}
\textit{Step 2.} The Euler--Lagrange system for \(y\) is
\[
\partial_i\!\left(
a_{ij}^{\alpha\beta}(y)\partial_jy^\beta
\right)
-
\frac{1}{2}
\partial_{y^\alpha}a_{ij}^{\beta\gamma}(y)
\partial_i y^\beta\partial_j y^\gamma
=0.
\]
On \(D\), the condition \(v=q\) gives
\[
y=0,
\quad
\partial_1y=\partial_2y=0.
\]
Since \(D\Psi(0)\) is injective, \(\partial_3v=0\) also gives
\[
\partial_3y=0.
\]
Hence
\[
Dy=0
\quad\text{and}\quad
a_{3j}^{\alpha\beta}(0)\partial_jy^\beta=0
\quad\text{on }D.
\]

\vspace{1em}
\textit{Step 3.} Extend \(y\) by zero across \(D\):
\[
\bar y(x)
=
\begin{cases}
y(x),&x_3\geq0,\\
0,&x_3<0
\end{cases}.
\]
The vanishing of \(y\), \(Dy\), and the co-normal derivative on \(D\) implies
that \(\bar y\) is a weak solution of the same elliptic system across the
plane.

\vspace{1em}
\textit{Step 4.} Difference-quotient estimates and elliptic bootstrapping make \(\bar y\) smooth. The analytic regularity theorem of Morrey~\cite{Morrey1958} implies
that \(\bar y\) is real analytic. Since \(\bar y=0\) in a lower half-ball, the
analytic identity theorem gives
\[
\bar y=0
\]
in a smaller full ball. Therefore \(v=q\) on a non-empty open subset of
\(B_1^+\).
\end{proof}

\subsection{Connectedness of the regular set}

\begin{lemma}
\label{lem:connected}
Let \(S\subset\mathbb R^3_+\) be relatively closed. If
\[
\dim_{\mathcal H}S<1,
\]
then
\[
\mathbb R^3_+\backslash S
\]
is path connected.
\end{lemma}

\begin{proof}
Fix \(a,b\in\mathbb R^3_+\backslash S\). Define the cones
\[
C_a
:=
\{a+t(z-a):z\in S,\ t\geq0\},
\quad
C_b
:=
\{b+t(z-b):z\in S,\ t\geq0\}.
\]
Their Hausdorff dimensions satisfy
\[
\dim_{\mathcal H}C_a,\,
\dim_{\mathcal H}C_b
\leq
1+\dim_{\mathcal H}S<2.
\]
Thus \(C_a\cup C_b\) cannot contain a three-dimensional ball. Choose
\[
c\in\mathbb R^3_+\backslash(C_a\cup C_b).
\]
Then neither segment \([a,c]\) nor \([c,b]\) meets \(S\). Hence
\[
[a,c]\cup[c,b]
\subset
\mathbb R^3_+\backslash S
\]
is a path from \(a\) to \(b\).
\end{proof}

\begin{remark}
This is the second place where dimension three is essential. The estimate
\(\dim_{\mathcal H}\operatorname{Sing}v<1\) ensures that the singular set
cannot disconnect the three-dimensional half-space.
\end{remark}

\subsection{Liouville rigidity from large-scale concentration}

\begin{proposition}[Liouville theorem for linearly growing minimizers]
\label{prop:liouville}
Let \(W(z,P)\) be autonomous, analytic, two-homogeneous, uniformly coercive,
and uniformly strictly convex in \(P\). Suppose that
\[
u\in H^1_{\mathrm{loc}}(\mathbb R^3_+;\mathbb S^2)
\]
is a class-A minimizer satisfying
\[
u=q
\quad\text{on }\partial\mathbb R^3_+.
\]
There is no such map satisfying
\begin{equation}
\varepsilon R
\leq
E_W(u;B_R^+)
\leq
MR
\quad\text{for any }R>0
\label{eq:liouville-linear-growth}
\end{equation}
for constants \(\varepsilon>0\) and \(M<+\infty\).
\end{proposition}

\begin{proof}
The proof is divided into several steps as follows.

\vspace{1em}
\textit{Step 1.} \emph{Large-scale compactness.} Choose \(R_j\to+\infty\) and define
\[
u_j(x):=u(R_jx).
\]
In dimension three,
\[
E_W(u_j;B_r^+)
=
R_j^{-1}E_W(u;B_{R_jr}^+).
\]
Thus \eqref{eq:liouville-linear-growth} gives
\[
\varepsilon r
\leq
E_W(u_j;B_r^+)
\leq
Mr
\quad r>0.
\]
Lemma~\ref{lem:package} and a diagonal argument yield a class-A minimizer
\(v\) such that
\[
u_j\to v
\quad\text{strongly in }
H^1_{\mathrm{loc}}(\mathbb R^3_+).
\]
Energy convergence gives
\begin{equation}
\varepsilon r
\leq
E_W(v;B_r^+)
\leq
Mr
\quad r>0.
\label{eq:limit-linear-growth}
\end{equation}
In particular, \(v\) is non-constant.

\vspace{1em}
\textit{Step 2.} \emph{Convergence of the boundary stresses.}
Let
\[
\mu_j:=\mu_{u_j}.
\]
Strong \(H^1_{\mathrm{loc}}\) convergence and the quadratic structure of the
stress imply
\[
T[u_j]\to T[v]
\quad\text{strongly in }L^1_{\mathrm{loc}}.
\]
Therefore, by \eqref{eq:boundary-stress-definition},
\[
\mu_j\rightharpoonup\mu_v
\]
vaguely as Radon measures on \(\mathbb R^2\).

\vspace{1em}
\textit{Step 3.} \emph{Concentration at the origin.}
Lemma~\ref{lem:flux} gives
\[
\mu_u(\mathbb R^2)<+\infty
\]
and, by the three-dimensional scaling law,
\[
\mu_j(A)=\mu_u(R_jA)
\]
for any Borel set \(A\subset\mathbb R^2\).

Let \(K\subset\subset\mathbb R^2\backslash\{0\}\). Then
\[
R_jK
\subset
\left\{
R_j\operatorname{dist}(K,0)
\leq | x'|
\leq
R_j\max_{y'\in K}| y'|
\right\}.
\]
These annuli escape to infinity. Since \(\mu_u\) is finite,
\[
\mu_j(K)=\mu_u(R_jK)\to0.
\]
Thus, for any \(\phi\in C_c(\mathbb R^2\backslash\{0\})\), with
\(K=\operatorname{spt}\phi\),
\[
\left\lvert\int\phi\,\mathrm{d} \mu_j\right\rvert
\leq
\|\phi\|_{L^\infty}\mu_j(K)
\to0.
\]
Passing to the vague limit gives
\begin{equation}
\operatorname{spt}\mu_v\subset\{0\}.
\label{eq:limit-stress-support}
\end{equation}

\vspace{1em}
\textit{Step 4.} \emph{Vanishing Cauchy data on a regular boundary patch.}
By Lemma~\ref{lem:package}, any boundary disk contains a relatively open
regular patch. Choose such a patch
\[
D\subset\subset\partial\mathbb R^3_+\backslash\{0\}.
\]
From \eqref{eq:limit-stress-support},
\[
\mu_v\llcorner D=0.
\]
Since \(v\) is smooth up to \(D\), Proposition~\ref{prop:stress} gives
\[
v=q,
\quad
\partial_3v=0
\quad\text{on }D.
\]
Lemma~\ref{lem:analytic} therefore yields
\[
v=q
\]
on a non-empty open subset of \(\mathbb R^3_+\).

\vspace{1em}
\textit{Step 5.} \emph{Propagation through the regular set.}
Let
\[
\mathcal R(v)
:=
\mathbb R^3_+\backslash\operatorname{Sing}v.
\]
The singular-set estimate
\[
\dim_{\mathcal H}\operatorname{Sing}v<1
\]
and Lemma~\ref{lem:connected} imply that \(\mathcal R(v)\) is path connected.
The map \(v\) is analytic on \(\mathcal R(v)\). Since it equals \(q\) on a
non-empty open subset, analytic continuation gives
\[
v=q
\quad\text{on }\mathcal R(v).
\]

\vspace{1em}
\textit{Step 6.} \emph{Contradiction.}
The singular set has three-dimensional Lebesgue measure zero. Hence
\[
v=q
\quad\text{almost everywhere in }\mathbb R^3_+,
\]
and therefore
\[
\nabla v=0
\quad\text{almost everywhere}.
\]
It follows that
\[
E_W(v;B_r^+)=0
\quad r>0,
\]
contradicting the lower bound
\[
E_W(v;B_r^+)\geq\varepsilon r
\]
in \eqref{eq:limit-linear-growth}.
\end{proof}

\begin{remark}
The limit \(v\) is not asserted to be homogeneous. The proof uses instead
\[
\mu_{u_R}(A)=\mu_u(RA)
\]
and the finiteness of \(\mu_u\). Under rescaling, the boundary stress vanishes
on compact subsets of \(\mathbb R^2\backslash\{0\}\), which replaces the
homogeneous-tangent-map argument of the classical harmonic-map theory.
\end{remark}

\subsection{Exclusion of boundary singularities}

\begin{proof}[Proof of Theorem~\ref{thm:main}]
Suppose, for contradiction, that \(n\) has a singular point
\[
a\in\partial\Omega.
\]
Flattening \(\partial\Omega\) near \(a\) and applying
Proposition~\ref{prop:tangent}, we obtain an autonomous, analytic, uniformly
coercive, and uniformly strictly convex quadratic density \(W_0\), together
with a class-A minimizer
\[
U:\mathbb R^3_+\to\mathbb S^2.
\]
Its trace on the flat boundary is a constant \(q\in\mathbb S^2\), and there
exist constants \(c,C>0\) such that
\[
cR
\leq
E_{W_0}(U;B_R^+)
\leq
CR
\quad\text{for any }R>0.
\]
The lower bound ensures that \(U\) is non-constant, while the upper bound gives
the linear growth required in Proposition~\ref{prop:liouville}. That
proposition rules out the existence of such a half-space class-A minimizer.
This contradiction proves
\[
\operatorname{Sing}n\cap\partial\Omega=\varnothing.
\]

Consequently, any point of \(\partial\Omega\) has a regular
half-neighborhood. Since \(\partial\Omega\) is compact, finitely many of these
neighborhoods cover the boundary. The small-energy regularity and elliptic
bootstrap contained in Proposition~\ref{prop:improvement} imply that \(n\) is
smooth on their union. Shrink the resulting boundary collar once so that its
closure in \(\overline\Omega\) remains inside this finite union. Hence there
exists an open neighborhood \(\mathcal U\) of \(\partial\Omega\) such that
\[
n\in C^\infty(\overline{\Omega\cap\mathcal U};\mathbb S^2).
\]
Thus \(n\) is smooth in a full boundary collar, completing the proof of
Theorem~\ref{thm:main}.
\end{proof}

\subsection{Verification for the Oseen--Frank energy}

We finally verify that the Oseen--Frank functional satisfies all structural
hypotheses used above.

\begin{proposition}[Oseen--Frank structural verification]

Let \(k_1,k_2,k_3>0\) and
\[
\alpha:=\min\{k_1,k_2,k_3\}.
\]
Then the following statements hold.

\begin{enumerate}[label=$(\theenumi)$]
\item The original Oseen--Frank energy and \(E_{W_\alpha}\) have the same
minimizers in any fixed Dirichlet class.

\item The density \(W_\alpha\) is autonomous, analytic, two-homogeneous,
uniformly coercive, and uniformly strictly convex in \(P\).

\item Boundary flattening preserves the uniform structural estimates, and the
rescaled densities converge to an autonomous density of the same type.

\item No smallness or closeness-to-isotropy condition is required on
\((k_1,k_2,k_3)\).
\end{enumerate}
\end{proposition}

\begin{proof}
If \(W_{\mathrm{OF}}\) denotes the original Oseen--Frank density, then
\(W_{\mathrm{OF}}\) and the coercive representative \(W_\alpha\) satisfy
\[
W_{\mathrm{OF}}(n,\nabla n)-W_\alpha(n,\nabla n)
=
\frac{1}{2}(k_2+k_4-\alpha)
\left\{
\operatorname{tr}((\nabla n)^2)-(\operatorname{div}n)^2
\right\}.
\]
The expression in the braces is a null Lagrangian. Its
integral depends only on the Dirichlet trace of \(n\), and is therefore
constant on any fixed Dirichlet class. Consequently,
\[
E_{\mathrm{OF}}(n;\Omega)
=\min_{v\in\mathcal A_g}E_{\mathrm{OF}}(v;\Omega)
\quad\Leftrightarrow\quad
E_{W_\alpha}(n;\Omega)
=\min_{v\in\mathcal A_g}E_{W_\alpha}(v;\Omega).
\]

Recall that
\[
\begin{aligned}
2W_\alpha(z,P)
={}&
\alpha| P|^2
+(k_1-\alpha)(\operatorname{tr}P)^2\\
&+(k_2-\alpha)(z\cdot\operatorname{curl}P)^2
+(k_3-\alpha)| z\times\operatorname{curl}P|^2,
\end{aligned}
\]
where
\[
\alpha=\min\{k_1,k_2,k_3\}>0.
\]
Since \(k_i-\alpha\geq0\), the last three terms are non-negative. Hence
\[
W_\alpha(z,P)\geq\frac{\alpha}{2}| P|^2.
\]
The corresponding upper bound follows from
\[
\lvert\operatorname{tr}P\rvert
+
|\operatorname{curl}P|
\leq C| P|,
\]
and therefore
\[
\frac{\alpha}{2}| P|^2
\leq
W_\alpha(z,P)
\leq
C(k_1,k_2,k_3)| P|^2.
\]
Moreover, the term \(\frac{\alpha}{2}| P|^2\) implies
\[
D_{PP}^2W_\alpha(z,P)\geq\alpha I,
\]
so \(W_\alpha\) is uniformly strictly convex in \(P\). Its explicit formula
also shows that it is autonomous, real analytic in \(z\), and
two-homogeneous in \(P\).

After flattening the boundary and rescaling, the transformed densities take
on the form
\[
W_j(x,z,P)
=
J_j(x)W_\alpha\bigl(z,PB_j(x)\bigr),
\]
where \(J_j\) is the Jacobian factor and \(B_j\) is induced by the inverse
differential of the flattening map. On compact sets,
\[
J_j\to J_0>0,
\quad
B_j\to B_0
\]
smoothly, with \(B_j\) and \(B_j^{-1}\) uniformly bounded. It follows that
the transformed densities retain uniform ellipticity and strict convexity.
Their blow-up limit is
\[
W_0(z,P)
=
J_0W_\alpha(z,PB_0),
\]
which is autonomous, analytic, two-homogeneous, uniformly coercive, and
uniformly strictly convex in \(P\).

All constants in these estimates may depend on
\(k_1,k_2,k_3\), but the argument uses only
\[
\alpha=\min\{k_1,k_2,k_3\}>0.
\]
No assumption that the Frank constants are equal or close to one another is
required.
\end{proof}

\begin{remark}[Where the sign enters]
The proof does not require a sign for
\[
\frac{\mathrm{d}}{\mathrm{d} R}\left(R^{-1}E_W(u;B_R^+)\right).
\]
The required positivity is supplied instead by the one-sided normal variation,
which produces the non-negative boundary stress measure \(\mu_u\).
\end{remark}

\section*{Acknowledgments}

This work is partially supported by the National Key R\&D Program of China
under Grant 2023YFA1008801.

The authors acknowledge the use of AI tools. All mathematical arguments and
proofs in the final manuscript were checked and written by the authors.

\bibliographystyle{plain}
\bibliography{OseenFrank_boundary_regularity}

\end{document}